\documentclass[final]{siamart190516}

\usepackage{amsmath,amssymb,amsfonts,amscd,bm,mathrsfs}

\numberwithin{equation}{section}

\usepackage{graphicx}

\graphicspath{{./figs/}}
\usepackage{graphicx,epsfig,epstopdf}
\usepackage{caption,subcaption}
\usepackage{enumitem}
\usepackage{xcolor}
\usepackage{algorithm}
\usepackage[noend]{algpseudocode}

\usepackage{hyperref}
\hypersetup{ %
    plainpages=false,
    linktocpage=true,   %
    linkcolor={red!50!black},
    citecolor={blue!50!black},
    urlcolor={blue!80!black}
}
\usepackage[hyperpageref]{backref}

\newcommand{\innerprod}[3][2]{\langle #2,#3 \rangle_{#1}}

\renewcommand{\a}{\alpha}

\newcommand{\vect}[1]{\bm{#1}}
\newcommand{\D}{\Delta}
\renewcommand{\d}{\delta}
\newcommand{\btau}{\vect{\tau}}
\newcommand{\G}{\Gamma}

\newcommand{\vn}{\vect{n}}
\newcommand{\vx}{\vect{x}}
\newcommand{\vy}{\vect{y}}
\newcommand{\vz}{\vect{z}}
\newcommand{\vB}{\vect{B}}
\newcommand{\cG}{{\mathcal{G}}}
\newcommand{\cP}{{\mathcal{P}}}

\newcommand{\cN}{{\mathcal{N}}}
\newcommand{\cT}{{\mathcal{T}}}

\newcommand{\RR}{\mathbb{R}}
\newcommand{\ZZ}{\mathbb{Z}}
\renewcommand{\SS}{\mathbb{S}}

\newcommand{\cross}{\!\times\!}

\newcommand{\wh}{\widehat}
\newcommand{\wt}{\widetilde}
\newcommand{\diam}{\operatorname{diam\,}}
\newcommand{\at}[1]{\big\vert_{\raisebox{-0.5pt}{\scriptsize$#1$}} }
\newcommand{\p}{\partial}

\newcommand{\abs}[1]{\left\vert {#1}\right\vert}
\newcommand{\dualp}[2]{{\left\langle {#1},{#2} \right\rangle}}

\newcommand{\norm}[1]{\left\Vert#1\right\Vert}
\newcommand{\snorm}[1]{\left\vert#1\right\vert}

\newcommand{\ljump}{\lbrack\hspace{-1.5pt}\lbrack}
\newcommand{\rjump}{\rbrack\hspace{-1.5pt}\rbrack}
\newcommand{\jump}[2]{\ljump {#1} 
\rjump_{\raisebox{-2pt}{\scriptsize$#2$}} }

\newcommand{\revision}[1]{\textcolor{black}{#1}}
\newcommand{\rrevision}[1]{\textcolor{black}{#1}}

\newtheorem{remark}[theorem]{Remark}

\title{A Bootstrap Multigrid Eigensolver}
\author{James Brannick\thanks{Department of Mathematics, Pennsylvania State 
University, University Park, State College, PA 16802, brannick@psu.edu} 
\and 
Shuhao Cao\thanks{Division of Computing, Analytics, and Mathematics, School of Science and Engineering, 
University of Missouri Kansas City, Kansas City, MO 64110, 
scao@umkc.edu}}

\begin{document}

\maketitle

\begin{abstract}
This paper introduces bootstrap multigrid methods for solving eigenvalue problems arising 
from the discretization of partial differential equations.  Inspired by the full 
bootstrap algebraic multigrid (BAMG) setup algorithm that includes an AMG eigensolver, it is illustrated how the algorithm
can be simplified for the case of a discretized partial differential equation (PDE), thereby developing a bootstrap geometric multigrid (BMG) approach.
We illustrate numerically the efficacy of the BMG method for: (1) recovering eigenvalues 
having large multiplicity, (2) computing interior eigenvalues, 
and (3) approximating shifted indefinite eigenvalue problems.    
Numerical experiments are presented to illustrate 
the basic components and ideas behind the success of the overall bootstrap multigrid approach. 
For completeness, we present a simplified error analysis of a two-grid bootstrap algorithm for the Laplace-Beltrami eigenvalue problem. 
\end{abstract}



\pagestyle{myheadings}
\thispagestyle{plain}
\markboth{JAMES BRANNICK AND SHUHAO CAO}
{BOOTSTRAP MULTIGRID EIGENSOLVER}

\section{Introduction}
\label{sec:intro}
The aim of this paper is to present a bootstrap multigrid framework for solving eigenvalue problems arising from discretizing partial differential equations.  The approach we consider is motivated by the bootstrap algebraic multigrid (BAMG) setup algorithm designed and analyzed in~\cite{BAMG2010,BAMG2,iBAMG,BAMGSchwinger,BAMGOpt}.  The key components of the BAMG setup algorithm that motivate our proposed geometric bootstrap multigrid (BMG) algorithm are the adaptive or bootstrap construction of the coarse spaces coupled with the BAMG multilevel eigensolver used in constructing interpolation (or prolongation).  The overall BAMG process simultaneously computes approximations of algebraically smooth error, and constructs improved or enriched coarse spaces 
using these approximations. The resulting BAMG setup algorithm has been shown to provide for an efficient algebraic multigrid (AMG) eigensolver as well as a robust and efficient AMG setup process for solving sparse linear systems~\cite{BAMG2010,BAMGSchwinger}. In this paper, we adapt the BAMG setup algorithm to a geometric method, i.e., one that makes explicit use of the PDE and discretization in its design, and draw connections between this BMG approach and existing geometric two-grid and multigrid eigensolvers that have been developed for approximating PDE eigenvalue problems. 
We illustrate numerically that the resulting bootstrap method is suitable for: (1) recovering eigenvalues having large multiplicity, (2) computing interior eigenvalues, and (3) approximating shifted indefinite eigenvalue problems.  

The problem of interest in this paper to present a BMG eigensolver is the Laplace-Beltrami eigenvalue problem
\label{sec:LB}
\begin{equation}
\label{eq:pb-eig-strong}
-\D_{\G} u = \lambda u,
\end{equation}
where $\D_{\G}$ denotes the Laplace-Beltrami operator on a 2-dimensional, smooth, 
orientable, and closed surface  $\G$, $\lambda\in \RR^+$ is the eigenvalue to the continuous 
eigenvalue problem, and
$u: \G \to \RR$ denotes the associated eigenfunction. Letting  
\begin{equation}
\label{eq:pb-bl}
a(u,v) := \int_{\G} \nabla_{\G} u\cdot \nabla_{\G} v \,dS,
\quad\text{and}\quad  b(u,v) := \int_{\G} u v\,dS, 
\end{equation} 
the weak formulation of~\eqref{eq:pb-eig-strong} is as follows: Find $u\in 
H^1(\G)$ and $\lambda\in \RR^+$ 
such that
\begin{equation}
\label{eq:pb-eig-weak}
a(u,v) = \lambda b(u,v), \quad \text{ for any } v\in H^1(\G),
\end{equation}
where $H^1(\G):= \left\{v \in L^2(\G): \nabla_{\G} v\in L^2(\G) \right\}$ equipped 
with following norms:
\begin{equation*}
\norm{v}_{H^1(\G)}^2  :=\norm{v}_{L^2(\G)}^2
+ \snorm{v}_{H^1(\G)}^2, \text{ with }
\snorm{v}_{H^1(\G)} := \norm{\nabla_{\G} v}_{L^2(\G)}.
\end{equation*}
Throughout this paper, we assume that a
surface finite element (SFEM) discretization \cite{Dziuk1988finite,Dziuk2013finite} 
is used to approximate problem \eqref{eq:pb-eig-weak}.

Though we derive the geometric BMG eigensolver with a focus on the finite element 
approximation to the problem above, we note that the overall strategy is applicable for much wider classes of PDE eigenvalue problems.  The Laplace-Beltrami model is selected as our model problem since it is a challenging problem in that its eigenvalues have a large multiplicity. Moreover, the Laplace-Beltrami spectrum on a 2-sphere is explicitly known, allowing us to study the proposed BMG approach in detail for a specific problem.

\subsection{The Bootstrap AMG setup}
\label{subsec:BAMG}
The BAMG setup algorithm that serves as the main motivation for our proposed BMG eigensolver can be viewed as an algebraic multigrid eigensolver for the algebraic system $A u = \lambda B u$, with $(\lambda,u)$ the unknown eigenpairs.  As we show below, it is straightforward to derive a BMG approach directly from the BAMG setup algorithm.  

In contrast to geometric multigrid methods that begin with a coarse problem (on a coarse mesh) and then aim to construct a sequence of problems on increasingly finer meshes (resulting more accurate approximations), in AMG the problem is assumed to be given on the finest mesh, and then is sought to be coarsened in an algebraic fashion, i.e., using only algebraic information available from the linear system.  Accordingly, we label the finest-grid system matrix as $A_0=A$.  Later on, in our description of the related geometric approaches, we adopt the usual geometric multigrid notation and set the coarsest-grid system to $A_0$.  When discussing two-level approaches, $A_h$ represents the fine-grid matrix, and $A_H$ denotes the coarse-grid system matrix, obtained using the Galerkin form $A_H = P^T A_h P$ in the AMG setting, or by rediscretizing in the geometric multigrid setting.  

Given the fine-grid system based on $A_h$, the BAMG setup algorithm aims to construct an interpolation operator $P$, and the
corresponding coarse-grid matrix $A_H = P^T A_h P$. Certain approximation property is satisfied by the interpolation 
operator and, hence, by the corresponding coarse space~\cite{BAMG2010}.  Generally, the AMG interpolation operator must accurately approximate 
(generalized) eigenvectors corresponding to small eigenvalues of the fine-grid system matrix $A_h$~\cite{BAMGOpt}, which are referred to as the 
slow-to-converge error of the AMG smoother.  Various techniques exist for constructing accurate interpolation for problems where 
such spectral information is available, e.g., Classical AMG~\cite{1985BrandtA_McCormickS_RugeJ-aa} and Smoothed Aggregation~\cite{PVanek_JMandel_MBrezina_1995a}.  The adaptive AMG~\cite{Breetal2,Breetal3} and BAMG algorithms were designed to 
construct $P$ for cases where this information is not available, e.g., quantum dynamics applications.  The main idea in these approaches is to make use of a multilevel algorithm in the AMG setup to compute slow-to-converge error of relaxation, and then to use these modes to adapt the coarse space. 
The overall BAMG approach for enriching the coarse spaces uses the combination of a multilevel eigensolver that efficiently and accurately computes approximations to the 
(generalized) eigenvectors with small eigenvalues of the fine-grid system matrix, i.e., it aims to solve problems of the form
\begin{equation}
A u = \lambda B u
\end{equation}
e.g., an SFEM discretization of problem \eqref{eq:pb-eig-weak}.
Then, a least squares process is used to construct interpolation to fit these computed approximations.  The overall approach is applied in the usual recursive way to obtain the corresponding multilevel algorithm.  
We describe the overall process adopting the same notation we used in~\cite{BAMG2010} and recall some numerical results obtained for the Poisson problem as well.  

Bootstrap AMG interpolation is derived to provide the best least squares (LS)
fit to a set of smooth test vectors $\mathcal{V} = \big\{v^{(1)}, \dots, v^{(k)}\big\}$.
Specifically, each row of interpolation, denoted by $p_i$,
is defined as the minimizer of the local least squares functional: 
\begin{equation}\label{eq:ls}
\mathcal{L}(p_i) = 
\sum_{\kappa=1}^k\omega_\kappa\left((v)_{\{i\}}^{(\kappa)} - \sum_{j\in \Omega_{i}} \left(p_{i}\right)_{j} v_{\{j\}}^{(\kappa)}\right)^{2}, 
\end{equation} 
where $\Omega_i$ are the sets of interpolation points, $v_{\widetilde{\Omega}}$
denotes the canonical restriction of a vector $v$ to the set $\widetilde{\Omega} \subset \Omega := 
\{1,..,n\}$, and $\omega_\kappa$ denote the {\em interpolation weights}.  A common choice for the weights is given by $\omega_k = \|v_k\|/\|Av_k\|$, in which case the LS functional can be viewed as a local version of the weak approximation property~\cite{BAMG2010}, assuming $A$ is symmetric and positive definite.

The rationale behind the bootstrap multilevel generalized eigensolver is as follows. If an initial multigrid hierarchy is constructed using LS interpolation, given the initial Galerkin operators $A_{0} = A, A_{1}, \ldots, A_{L}$ on each level
and the corresponding interpolation operators $P_{l+1}^{l}, l =
0,\ldots,L-1$, define the composite interpolation matrices and corresponding mass matrices as
$P_{l} = P_{1}^{0}\cdot \ldots \cdot P_{l}^{l-1}$ and $T_{l} = P_{l}^{T}B_0P_{l}$, $l = 1, \ldots, L$,  then for any level $l$ and any given vector $x_{l} \in \mathbb{R}^{n_l}$ and $\lambda^{l}\in \mathbb{R}$ such that
$  A_l x_{l} = \lambda^{l}T_l x^{l}$ we have
\begin{equation}\label{eq:evalapprox}
 \text{Rayleigh quotient of } P_l x^{l} :=  \frac{\innerprod[A_{l}]{x_{l}}{x_{l}}}{\innerprod[T_{l}]{x_{l}}{x_{l}}} = \frac{\innerprod[A_0]{P_{l}x_{l}}{P_{l}x_{l}}}{\innerprod{P_{l}x_{l}}{P_{l}x_{l}}} = \lambda^{l},
\end{equation} 
 where we used that $\innerprod[A_{l}]{x_{l}}{x_{l}} = \innerprod[A_0]{P_{l}x_{l}}{P_{l}{ x_{l}}}$, which follows from the definition of
$ A_{l}= P_{l}^T A_0 P_{l}$.
This result relates the eigenvectors and eigenvalues
of any of the coarse-grid matrices to 
the eigenvectors and eigenvalues of the finest-grid operator $A$.
Note that the eigenvalue approximations in~\eqref{eq:evalapprox} are continuously updated within the algorithm so that the overall approach resembles an inverse Rayleigh-Quotient
iteration found in eigenvalue computations (cf.~\cite{Wilkinson:1988:AEP:59657}).  
Fig.~\ref{fig:boot:setupcycle} provides a schematic outline of a $V$-cycle version of the BAMG setup algorithm.  We note that $\mathcal{V}^r$ and $\mathcal{V}^e$ are the sets of approximate test vectors from relaxations and approximate eigenvectors, respectively, that are used in computing the LS interpolations.

\begin{figure}
\begin{center}
\includegraphics[width=\textwidth]{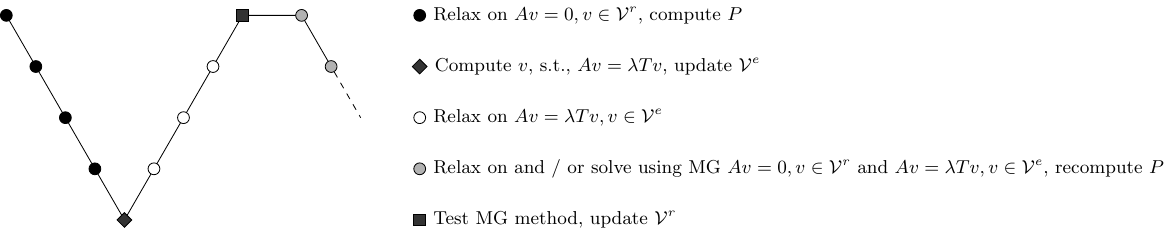}
  \caption{A Possible Galerkin Bootstrap AMG $V$-cycle setup scheme from~\cite{BAMG2010}.}
\label{fig:boot:setupcycle}
\end{center}
\end{figure}

Following are a few general remarks regarding the above BAMG setup algorithm and its conversion to a geometric multigrid method, before we proceed with a general discussion of the BMG eigensolver:

\begin{itemize}[topsep=2pt, leftmargin=1.5em]

\item The BAMG eigensolver only solves the eigenvalue problem on the coarsest-level and then uses an iterative method to approximately solve (simpler) shifted linear systems on the corresponding finer levels.  This strategy of only solving the eigenvalue problem on the coarsest level is what gives the approach its efficiency and ultimately its near optimality.

\item The method continuously updates (enriches) the coarse space using these improved eigen-approximations in the LS interpolation process and then the Galerkin definition of the coarse-level operator.  Moreover, it is this enrichment that gives the method the flexibility 
to construct approximations to both oscillatory and smooth eigenvectors using very coarse meshes.

\item The overall procedure uses a relatively large enrichment space since experimentally this gives the best overall results in the BAMG construction of interpolation and, hence, ultimately yields the best overall AMG solver. 
We propose a similar enrichment strategy that uses multiple vectors in the BMG eigensolver,    however, since we are interested in solving the eigenvalue problem and not using the approach as a setup process for defining interpolation, 
we omit the set of relaxed vectors $\mathcal{V}^e$ in the BMG algorithm and enrich the coarse space using only eigenvector approximations $\mathcal{V}^e$.  Correspondingly, the BMG approach begins by solving an eigenvalue problem on the coarsest level $A_0$.
\item A main cost in the BAMG setup is the need to compute the AMG hierarchy in each iteration, i.e., the cost of computing
 $P_{l+1}^{l}$, and $A_l  = ( P_{l+1}^{l})^T A_{l-1} P_{l+1}^{l}$ on all levels.  In the BMG setting, the coarse spaces are defined 
 using the chosen finite element discretization plus an additional global enrichment space.  This difference is, in fact, the main modification that we make to the BAMG eigensolver in order to arrive at the BMG approach.  
 \item Though we include the possibility of using the existing MG solver in the BAMG setup when solving the shifted linear systems on finer meshes, typically only a few relaxation steps are needed.  We use many pre- and post-smoothing steps in the BAMG setup 
 in order to reduce the number of setup cycles which, in turn, offsets the cost of recomputing the AMG hierarchy in each iteration of the BAMG setup. For example, using a standard discretization for the 2D Laplace eigenvalue problem in \cite{BAMG2010} on a structured mesh, two $V(4,4)$ BAMG setup cycles as depicted in Fig.~\ref{fig:boot:setupcycle} on a finest-level mesh of size $64 \times 64$ using standard (full) coarsening and Gauss-Seidel relaxation yield comparable result versus direct methods with all geometric information. These partial experimental results from~\cite[Fig. 4.2 and Table 4.5]{BAMG2010} illustrate the efficacy and potential of the BAMG setup algorithm as an AMG eigensolver.
\end{itemize}

\subsection{Existing geometric multigrid eigensolvers}

Generally speaking, in situations where multiple 
eigen-modes are required, a highly refined uniform mesh is needed for a good approximation for all these modes. In such cases, solving the associated systems of equations on this given fine mesh can be prohibitively expensive. There are many methods taking advantage of a multilevel hierarchy to remedy this cost. For the 2D Laplace eigenvalue problem on the plane, a 
two-grid eigensolver that reduces the overall costs of solving the resulting finite element systems was proposed in~\cite{XuZhou99eig}, and later improved using a 
Newton-type iteration in~\cite{HuCheng11ieig}. A similar 
approach was designed for the Maxwell eigenvalue problem in~\cite{ZhouHul14MaxEig}. These 
two-grid methods involve a coarse mesh and a fine mesh and the finite element spaces 
defined on these meshes. In addition, a direct solve, e.g., \texttt{eig} in 
MATLAB, is used to solve the coarse space eigenvalue problem, and then Newton's 
method is applied on the fine mesh in order to solve the 
nonlinear eigenvalue problem using the coarse solution as an initial guess, which results solving a linear source problem. The a priori analysis for the finite element approximations to the Laplace-Beltrami eigenvalue problem is studied in \cite{Bonito2018priori,Reusken2022analysis}.
To the best of our knowledge, there is no known two-grid (or multigrid) 
eigensolver for the Laplace-Beltrami operator on surfaces in the finite element 
setting, which is the model problem we design the BMG eigensolver for in this paper.  

Though two-grid methods do provide significant improvements when compared with 
single-grid methods (such as the Arnoldi algorithm) in terms of their computational 
complexity, 
they too have drawbacks in practice. First, 
two-grid methods are generally not optimal since the mesh spacing tends to zero even for
the coarse eigenvalue problem, which needs to be solved with high accuracy. 
For example, in order to resolve eigenpairs corresponding to large eigenvalues 
in the discrete spectra for our model problem, the coarse mesh used in the two-grid method must be fine 
enough to capture these oscillatory modes. In practice, we observe a ``loss of spectra'' phenomenon when using two-grid methods 
where the coarse mesh is not fine enough. This issue is overcome in our proposed BAMG algorithm by 
using the idea to enrich the coarse space as in~\cite{BAMG2010} and as discussed above. We note in addition 
that, these two-grid eigensolvers often require solving a linear source problem on the 
fine mesh that is indefinite so that using optimal solvers such as multigrid can 
become problematic. As we show 
numerically in this paper, it is not necessary to solve this indefinite problem 
directly, and a few sweeps of an iterative solver suffice to obtain an almost optimal 
multigrid algorithm.  In fact, for certain cases, we show that the shift can be 
moved to the right hand side using an interpolated coarse approximation to the 
eigenfunction of interest.

In~\cite{CaiMandelMcCor97mgeig},
multilevel analogues of the two-grid solvers noted above are developed.  
Specifically, the paper develops multilevel approaches for nearly singular elliptic 
problems and eigenvalue problems. It should be noted that these methods 
are able to approximate the components in the 
eigenspace with small eigenvalues of \eqref{eq:pb-eig-weak} and 
as presented cannot be used to approximate larger eigenpairs.
This is because on the coarsest level the corresponding space is defined using
standard finite element basis functions, e.g., nodal Lagrange basis, and thus 
it is not possible to approximate oscillatory functions.

In a recent paper \cite{Lin2015multi}, another multilevel approach 
was developed in which the coarse eigenvalue problem is solved in an enriched space.  
This enrichment is achieved by including a single extra function in the coarse space 
that is obtained by solving a positive definite source problem on a finer mesh. Then, 
this two-grid correction scheme is used repeatedly to span multiple levels, 
resembling the bootstrapping procedure developed in~\cite{BAMG2010} but 
only for a single enrichment vector. The approach we 
propose involves a suitable geometric projection from
coarse spaces to fine spaces (defined on a sequence of refined and non-nested meshes)
and the use of a bootstrap enrichment procedure to iteratively improve the coarse spaces until 
the desired approximation is computed to sufficient accuracy.

It is known that the Laplace-Beltrami eigenvalues have very high 
multiplicity on closed surfaces (e.g., see \cite[Chapter 3]{Schoen-Yau}), which adds to the difficulties associated with solving this system.  
For example, for the Laplacian-Beltrami operator on the 2-sphere, the 
number of linearly independent eigenfunctions associated with $l$-th distinct 
eigenvalue $\lambda = l(l+1)$ is $2l+1$. Thus, the bootstrap approach we propose
which enriches the coarsest space with a subspace of linearly independent approximations 
is of particular interest for this model problem. 
We note that the idea to enrich the coarse space in designing eigensolvers goes back 
to~\cite{KZ15} and~\cite{M15}.  
These authors also analyze an iterative method for computing the smallest eigenpair 
under the somewhat restrictive condition that the initial guess
of the eigenfunction is sufficiently close to the smallest one, namely that its 
Rayleigh quotient lies between the smallest and second smallest
eigenvalues. In~\cite{Chan15}, the method from~\cite{KZ15} is extended to both
two-grid and multigrid methods and an algorithm for computing a given number of the 
smallest eigenpairs is presented. The paper also presents a convergence theory with 
less restrictive assumptions on the initial guess.   

In this paper, we develop the BMG eigensolver for the surface finite element discretization of the shifted 
Laplace-Beltrami eigenvalue problem. 
The base two-grid method can be viewed as a generalization of the approaches 
proposed in~\cite{KZ15,M15,Chan15,Lin2015multi} in that the coarse space is enriched 
with a subspace, instead of a single eigenfunction and we consider computing 
interior eigenvalues directly by introducing a shift. Alternatively, our proposed approach 
can be viewed as a simplification of the BAMG algorithm~\cite{BAMG2010} in that we 
use the finite element spaces to explicitly define the components of the multilevel 
method, including interpolation and restriction operators among different levels, 
and the enriched coarsest space eigenvalue problem. \revision{Meanwhile, the iterative procedure involving multi-levels of mesh refinement resembles the geometric cascadic multigrid in \cite{Urschel;Hu;Xu;Zikatanov:2014cascadic} and \cite{Han;Xie;Xu:2017CASCADIC}.}  Overall, we note that though we focus on this model problem the approach we present here
is applicable to much wider classes of problems with relatively few modifications.  

This paper is organized as follows. In Section 2, we provide some preliminary 
notations and present the a priori estimate of the 
surface finite element method for the eigenvalue problems from 
\cite{Bonito2018priori}. In Section 
3, we introduce the standard two-grid method for Laplace-Beltrami eigenvalue problems on 
surfaces mimicking the approaches developed for elliptic eigenvalue problem 
in~\cite{XuZhou99eig,HuCheng11ieig,ZhouHul14MaxEig}.
In addition, we prove the convergence of this method for the case of a smooth, 
closed, and orientable surface. In Section 4, we derive the 
finite element bootstrap multigrid method for the Laplace-Beltrami eigenvalue problem and give details on the approach. In addition,
we prove the convergence of the bootstrap two-grid eigensolver for the shifted Laplace-Beltrami eigenvalue problem in the case the enrichment space is defined using several functions.  We note that on a surface the non-nestedness of the meshes also introduces a geometric error into the discrete approximations.  In this regard, our model problem presents several challenges and our  two-grid approach is quite general since the inclusion of a shift also covers the case of computing interior eigenvalues (with high multiplicity).  Section 5 contains 
results of numerical experiments for both the two-grid and 
multigrid methods applied to the model problem on $\SS^2$. Note that 
by fixing the geometry we are able to study the algorithm
in a detailed and systematic way.

\section{Notation and preliminary results}
\label{sec:prelim}
In this section, the finite element approximation, together 
with its a priori error estimate, to the eigenvalue problem \eqref{eq:pb-eig-weak} are
presented. 
Before presenting the details of the discretization we set notation that is used throughout 
the paper.  Here, we need to distinguish between the continuous solution of the eigenvalue problem
on the continuous surface, $\Gamma$, the continuous solution on the discrete surface (mesh), $\Gamma_h$, used to approximate
the surface (which introduces a geometric error), and the discrete finite element approximation obtained 
on the discrete surface which admits a discretization error.  

We denote the eigen-modes to the continuous eigenvalue problem on the continuous surface by lower case letters, e.g., $u$ and the associated eigenvalues by $\lambda$.  The associated solution to the continuous eigenvalue problem
on the discrete surface $\G_h$ is denoted by $\bar{u}$ and the associated 
eigenvalues are given by $\bar{\lambda}$.  We distinguish between these two 
solutions since the solution on the discrete surface will 
only approximate the true solution since the discrete surface approximation 
introduces a geometric error. The finite element spaces are similarly defined, where 
we use instead capital calligraphic letters, e.g., the finite element space on 
$\G_h$ is denoted by $\mathcal{V}_h$.

In describing a multigrid approach for solving eigenvalue problems it is convenient to 
distinguish between the solutions of the non-linear finite-element eigenvalue problem and 
the solutions to associated linear problems.
Assume that a pair of finite element spaces are given, where the 
coarse finite element approximation space is $\mathcal{V}_H$ and the fine 
space is $\mathcal{V}_h$. 
Then, the subscript on $u_h$ denotes that this is the solution to a direct 
eigensolve of the eigenvalue problem in $\mathcal{V}_h$. 
A superscript $u^h$ denotes a source problem approximation in the fine 
space $\mathcal{V}_h$. In a two-grid method, a direct eigen-solve solution $u_H$ in 
a coarser space $\mathcal{V}_H$ (to a nonlinear problem) is used, and a source 
problem approximation $u^h$ using $u_H$ as data in a finer space $\mathcal{V}_h$ 
yields an improved approximation. The vector representation of $u_h$ in the 
canonical finite element basis of $\mathcal{V}_h$ is denoted by $U_h$, while $U^{h}$ is for 
$u^h$. This nomenclature, where a subscript 
corresponds to a direct eigensolve and a superscript stands for source problem 
approximation, is adopted throughout the paper. Matrices and operators are denoted 
with capital letters, where the actual definition should be clear from context. 
Moreover, the details on our construction 
of fine meshes and the associated eigen-spaces are given in the beginning of Section~\ref{ssec:tgES}.

We use $x \lesssim y$ and $z \gtrsim w$ 
to represent $x \leq c_1 y$ and $z \geq c_2 w$ respectively, where $c_1$ and $c_2$ 
are two constants independent of the mesh size $h$ and eigenvalues. The constants in 
these inequalities may in certain cases depend on specific eigenvalue(s) and when 
such dependence exists, it will be stated explicitly.

The surface gradient operator on a 2-dimensional smooth orientable surface that can 
be embedded into $\RR^3$ can be defined using extensions $\nabla_{\G}: H^1(\G) \to 
(L^2(\G))^3$ as follows.
\revision{
\begin{equation}
\label{def:sg}
(\nabla_{\G} f)(\vx) := 
\bigl(I - \vn(\vx)\vn(\vx)^{\top}\bigr) \nabla\wt{f}(\bm{x}) 
=\vn(\vx)\cross\bigl(\nabla\wt{f}(\vx)\cross\vn(\vx)\bigr),
\end{equation}
where $\wt{f}$ is a smooth extension of $f$ to a 3-dimensional tubular neighborhood $U$ of $\G$, 
$\nabla: H^1(\Omega) \to (L^2(\Omega))^3$ 
is the weak gradient operator in $\RR^3$, and $\vn(\vx)$ is the unit normal 
pointing 
to the outside of this closed surface at point $\vx$.
The Laplace-Beltrami operator $\D_{\G}$ is then defined in a distributional sense:
\begin{equation}
\label{eq:bl-cont}
\dualp{-\D_{\G} f}{g}=\int_{\G} (-\D_{\G} f)g\,dS:= 
\int_{\G} \nabla_{\G} f\cdot \nabla_{\G} g \,dS,\quad \forall g\in C^{\infty}(\G).
\end{equation}
} For a more detailed definition and the technicalities that arise when
defining a differential operator on surfaces, we refer the readers to 
\cite{Demlow2007adaptive,Dziuk2013finite,BonitoDemlowNochetto2020finite}.

\subsection{The eigenvalue problem on the discrete surface}
Let $\cT_h = \{T\}$ be a triangulation and $\G_h = \cup_{T\in \cT_h}T$ be piecewise 
planar surface approximating the continuous surface $\G$, where 
$T$ stands for the ``flat'' triangular element. $\G_h$ is assumed to be 
quasi-uniform and regular. The mesh size is then defined as the maximum of the 
diameter of all the triangles: $h:=\max_{T\in \cT_h}\diam T $.  Furthermore, the set 
of all vertices is denoted by $\cN_h$. For any $\vz\in \cN_h$, it is assumed that 
$\vz\in \G$, i.e., any vertex in the triangulation lies on the original continuous 
surface $\G$.  

Note that the surface gradient on a smooth surface carries over naturally to a 
discrete surface $\G_h$ (e.g., see \cite{Wei2010recovery}): the 
unit normal $\vn(\vx)$ is now a constant vector $\vn_T$ for each point $\vx\in T$. 
For ease of notation, the surface gradient 
$\nabla_{\G_h}$ on $\G_h$ and on $\G$ will both be denoted by $\nabla_{\G}$, where 
the definition should be clear from the context.  With these definitions the 
bilinear forms on $\G_h$ are as follows
\begin{equation}
\label{eq:bl-dis}
a_h(\bar{u},\bar{v}):= \int_{\G_h} \nabla_{\G} \bar{u} \cdot \nabla_{\G} \bar{v} \,dS,
\;\text{ and }\; b_h(\bar{u},\bar{v}):= \int_{\G_h}  \bar{u} \bar{v} \,dS.
\end{equation}
Note that, the subscript used in defining the bilinear forms is used to represent that the
continuous problem has been restricted to a discrete surface, $\G_h$.
Moreover, the fact that this discrete surface $\G_h$ is piecewise linear affine, which gives a 
$C^{0,1}$-surface, implies that the Sobolev space $H^1(\G_h)$ is well-defined
(see \cite{Dziuk1988finite}). 

The weak formulation for the 
eigenvalue problem on the discrete surface $\G_h$ is now given by:
find $\bar{u}\in H^1(\G_h)$ and $\bar{\lambda} \in \RR^+$ such that
\begin{equation}
\label{eq:pb-eig-dis}
a_h(\bar{u},\bar{v}) = \bar{\lambda}\, b_h(\bar{u},\bar{v}), \quad \text{ for any } \bar{v}\in  H^1(\G_h).
\end{equation}

Using the Poincar\'{e} inequality 
(see \cite{Demlow2007adaptive} Lemma 2.2) or the compact embedding of $H^1(\G_h)/\RR 
\subset \subset L^2(\G_h)$ when $\G_h$ is a
piecewise linear affine manifold (\cite[Chapter 2]{aubin2013} ), and the geometric 
error estimate between \eqref{eq:pb-bl} and \eqref{eq:bl-dis} (e.g. see 
\cite[Section 4]{Dziuk2013finite} ), it follows 
that if the mesh is sufficiently fine (required for the coercivity), then for any 
$\bar u, \bar v\in H^1(\G_h)/\RR$
\begin{equation}
\label{eq:pb-wp}
a_h( \bar u, \bar v) \lesssim \norm{ \bar u}_{H^1(\G_h)} \norm{ \bar 
v}_{H^1(\G_h)},\;
\text{ and }\; a_h( \bar u, \bar u) \gtrsim \norm{ \bar u}_{H^1(\G_h)}^2.
\end{equation}
If $ \bar v\neq 0$, $b_h( \bar v, \bar v) = \norm{v}_{L^2(\G_h)}^2 > 0$, the 
coercivity and continuity of \eqref{eq:pb-wp} implies that $a_h(\cdot,\cdot)$ induces a bounded, compact, and 
self-adjoint operator. 
By the 
Hilbert-Schmidt theory, and the spectrum theory of the Laplacian-Beltrami operator 
on compact surfaces (every closed surface being compact, see \cite{chavel1984}), 
problem \eqref{eq:pb-eig-dis} is a well-posed self-adjoint eigenvalue problem. The 
eigenvalues $\{\bar{\lambda}_i\}_{i=0}^{\infty}$ for problem \eqref{eq:pb-eig-dis} 
form a discrete sequence, starting from 0, with no 
accumulation point:
\[
0=\bar{\lambda}_0< \bar{\lambda}_1 \leq \bar{\lambda}_2 \leq \dots \to \infty.
\]
Moreover, the eigenfunctions $\phi_k$ associated with $\bar{\lambda}_k$ are
orthogonal in the sense that $b_h(\phi_i,\phi_j) = \d_{ij}$.

Let $M(\bar{\lambda})$ be the eigenspace spanned by the eigenfunctions 
associated with $\bar{\lambda}$ for \eqref{eq:pb-eig-dis} defined on the discrete 
surface $\G_h$:
\begin{equation}
M(\bar{\lambda}) := 
\{ \bar u \in H^1(\G_h): a_h( \bar u, \bar v) = \bar{\lambda} b_h( \bar u, \bar v), \forall  \bar v\in H^1(\G_h) \}.
\end{equation}
Similarly, the eigenspace for the continuous eigenvalue problem \eqref{eq:pb-eig-weak} on 
$\G$ is given by
\begin{equation}
M({\lambda}) := 
\{u \in H^1(\G): a(u,v) = {\lambda} b(u,v), \forall v\in H^1(\G) \}.
\end{equation}

\subsection{Finite element approximation}
In this subsection, the surface finite element 
discretization \eqref{eq:pb-eig-fea} of eigenvalue problem \eqref{eq:pb-eig-dis} is 
established. Here, if the 
geometric error introduced by the discrete surface is sufficiently small, then the
surface finite element approximates the eigenvalue problem \eqref{eq:pb-eig-weak} on the 
original smooth surface. In the last part of this subsection, the a priori error 
estimation for the surface finite element eigenvalue problem using a direct eigensolve is 
presented, giving Lemma \ref{lem:apriori}. Note that, the orders of the 
approximation errors for the 
computed eigenpairs given in this lemma are useful in determining the effectiveness of an
iterative procedure to obtain approximate eigenpairs, namely, the two-grid or multigrid method 
need to compute approximations with the same (or higher) order of approximation error.

The finite element approximation to problem \eqref{eq:pb-eig-dis} $\mathcal{V}_h$ is 
as follows:
\begin{equation}
\label{eq:sp-fe}
\mathcal{V}_h = \{\phi_h \in C^0(\G_h): \phi_h\at{T} \in P^1(T) \;\; \forall \,T\in 
\cT_h\}.
\end{equation}
The discretization to problem \eqref{eq:pb-eig-dis} is: to find $u_h\in 
\mathcal{V}_h$ and $\lambda_h\in \RR^+$ such that
\begin{equation}
\label{eq:pb-eig-fea}
a_h(u_h,v_h) = \lambda_h b_h(u_h,v_h), \quad \forall\, v_h\in \mathcal{V}_h.
\end{equation}
Note that the finite element approximation problem \eqref{eq:pb-eig-fea} 
serves as a straightforward 
conforming discretization to \eqref{eq:pb-eig-dis} on the discrete polygonal 
surface $\G_h$, but not directly to the original eigenvalue problem 
\eqref{eq:pb-eig-weak} on $\G$. 

The connection between the approximation on $\G_h$ and its continuous 
counterpart on the surface $\G$ is established through a bijective lifting operator 
(see \cite{Demlow2007adaptive}), between any triangle 
$T\subset\G_h$ to a curvilinear triangle on $\G$. Then, for any $v \in H^1(\G_h)$, 
its lifting $\wt{v}$ to the continuous surface $\G$ can 
be defined as follows: for any point $\vx\in \G_h$, there is a unique point 
$\wt{\vx}\in \G$, such that
\begin{equation}
\label{eq:lift}
\wt{\vx} + d(\vx)\vn(\wt{\vx}) = \vx,\; 
\text{ and }\;\wt{v}\bigl(\wt{\vx}\bigr) = v(\vx),
\end{equation}
where $d(\vx)$ is the signed distance to $\G$ at point $\vx\in \G_h$
and $d(\vx)$ is positive when $\vx$ is outside of the closed surface $\G$, with 
$|d(\vx)| = \min_{\vy\in \G}|\vx - \vy|$. 
\begin{lemma}[Lemma 3 in \cite{Dziuk1988finite}, Lemma 4.7 in \cite{Dziuk2013finite}]
\label{lem:lift}
If $d$ defined in \eqref{eq:lift} satisfies $\norm{d}_{L^{\infty}(\G_h)} \lesssim h^2$, then for any $v\in H^1(\G_h)$
\begin{equation}
\label{eq:err-geom}
\begin{aligned}
(1-ch^2) \snorm{v}_{{H^1(\G_h)}} \leq \snorm{\wt{v}}_{{H^1(\G)}} \leq (1+ch^2) 
\snorm{v}_{{H^1(\G_h)}},
\\ 
\text{and } \; (1-ch^2) \norm{v}_{{L^2(\G_h)}} \leq \norm{\wt{v}}_{{L^2(\G)}} \leq (1+ch^2) 
\norm{v}_{{L^2(\G_h)}}.
\end{aligned}
\end{equation}
For any $u_h,v_h\in \mathcal{V}_h$,
\begin{equation}
  \label{eq:err-geom-2}
|b(\wt{u}_h, \wt{v}_h) - b_h(u_h,v_h)|\leq c h^2 \norm{u_h}_{L^2(\G_h)} \norm{v_h}_{L^2(\G_h)}.
\end{equation}
\end{lemma}
When using the linear surface finite element on a piecewise planar mesh on $\G_h$, 
the following estimate holds (\cite{Bonito2018priori}):
\begin{lemma}
\label{lem:apriori}
If the mesh size $h$ is small enough and all the vertices of $\G_h$ lie on $\G$, 
then for an eigenvalue $\lambda$ of problem \eqref{eq:pb-eig-weak} with multiplicity 
$m$ on $\G$, there exist $m$ $\lambda_{h,k}$'s that are the eigenvalues of problem 
\eqref{eq:pb-eig-fea} on $\G_h$, and 
\begin{equation}
\label{eq:est-ev}
|\lambda_{h,k} - \lambda |\leq C(\lambda) h^2, \quad \forall 1\leq k\leq m.
\end{equation}
Moreover, let 
$M(\lambda_h) = \operatorname{span} \left\{ u_{h,k} \right\}_{k=1}^m$, where 
$u_{h,k}$'s are the eigenfunctions associated with 
$\lambda_h$, then for any eigenfunction $u\in M({\lambda})$,
\begin{equation}
\label{eq:est-ef}
\min_{w_h\in M(\lambda_h)}\vert{u - \wt{w}_h}\vert_{H^1(\G)} \leq C(\lambda) h.
\end{equation}
\end{lemma}

\section{A two-grid eigensolver}\label{ssec:tgES}

In this section, a two-grid algorithm to accelerate the approximation of 
\eqref{eq:pb-eig-fea} exploiting a coarse-fine hierarchical structure is presented 
in Algorithm \ref{alg:tg}. Algorithm \ref{alg:tg} is an extension to the surface 
case of the two-grid methods introduced 
in~\cite{HuCheng11ieig,ZhouHul14MaxEig,XuZhou99eig}. 

Assume that a pair of near hierarchical meshes are given, where the fine mesh 
$\cT_h$ is uniformly refined from the coarse mesh $\cT_H$. The 
newly created vertices, which are the midpoints of the edges of $\cT_H$, are 
projected onto the continuous surface $\G$. Denote the finite element approximation 
space \eqref{eq:sp-fe} on $\cT_H$ by $\mathcal{V}_H$, and the fine 
space by $\mathcal{V}_h$.

Under this setting, we construct a two-grid solution $u^h$ to approximate $u_h$ that 
is sought through a direct solve. First the subscript $u_H\in\mathcal{V}_H$ is 
obtained by a direct eigensolve of eigenvalue problem \eqref{eq:pb-eig-fea} (a nonlinear 
problem). Then the superscript $u^h$ in \eqref{eq:pb-f-s} of Algorithm \ref{alg:tg}
is obtained by a source problem approximation using $u_H$'s prolongation as the 
right hand side data in the finer space $\mathcal{V}_h$.

An important difference in the surface case as considered in this paper, when being 
compared with previous works for the Laplacian eigenvalue problem on the plane, is that the projection operators need extra 
care.  When the mesh is refined, the finite element space on the coarse 
mesh is not a subspace of the fine mesh. The natural inclusion $\mathcal{V}_H 
\not\subset \mathcal{V}_h$ does not hold (see Figure \ref{fig:tg-g}).

\begin{figure}[htbp]
  \centering
\begin{subfigure}[b]{0.3\linewidth}
    \centering
    \includegraphics[width=0.9\textwidth]{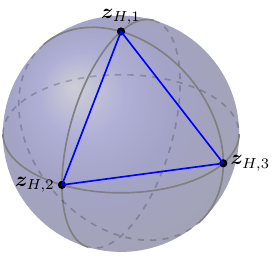}
    \caption{\label{fig:tg-g1}}
\end{subfigure}%
\hspace{0.1in}
\begin{subfigure}[b]{0.3\linewidth}
      \centering
      \includegraphics[width=0.9\textwidth]{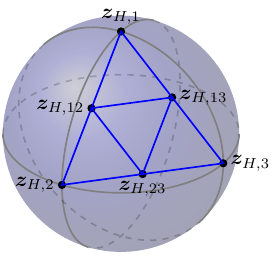}
      \caption{\label{fig:tg-g2}}
\end{subfigure}
\hspace{0.1in}
\begin{subfigure}[b]{0.3\linewidth}
      \centering
      \includegraphics[width=0.9\textwidth]{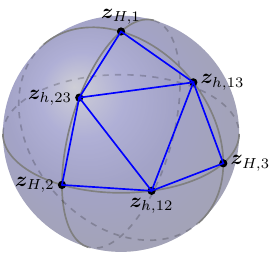}
      \caption{\label{fig:tg-g3}}
\end{subfigure}
\vspace{-0.1in}
\caption{The illustration of the prolongation operator from 
the coarsest octahedral mesh to one level finer on a 2-sphere: (\subref{fig:tg-g1}) Blue triangle is an element on the coarse mesh to approximate the geometry 
of a 2-sphere. A function $w\in \mathcal{V}_H$ has value $w(\vz_{H,i})$ at $\vz_{H,i}$, 
$i=1,2,3$. (\subref{fig:tg-g2}) Uniform refinement by connecting 3 midpoints of each edge. Standard 
prolongation yields the prolongated function's value at $\vz_{H,ij}$ which is 
$(w(\vz_{H,j})+ w(\vz_{H,i}))/2$. (\subref{fig:tg-g3}) Newly created vertices $\vz_{H,ij}$ being projected onto the surface as the vertices for the fine mesh. }
\label{fig:tg-g}
\end{figure}

Consider the geometric projection operator $\cG_{H\to \wh{H}}$ such that 
\begin{equation}
\label{eq:proj}
\wh{V}_{{H}} := \cG_{H\to \wh{H}} \mathcal{V}_H \subset \mathcal{V}_h.
\end{equation}
The definition of $\cG_{H\to \wh{H}}$ is then given by: for $\vz\in 
\cN_h$ and $\wh{w}_H\in \mathcal{V}_h$,
\begin{equation}
\label{eq:proj-1}
\wh{w}_H(\vz) = \cG_{H\to \wh{H}}w_H(\vz) = 
\begin{cases}
w_H(\vz) & \text{ if } \vz\in \cN_H,
\\[3pt]
\bigl(w_H(\vz_{H,1}) + w_H(\vz_{H,2})\bigr)/2 & 
\text{ if } \vz
\in \cN_h\backslash \cN_H,
\end{cases}
\end{equation}
where in the second case, $\vz$ is the projected midpoint 
between the coarse mesh vertices $\vz_{H,1}$ and $\vz_{H,2}$ using the lifting map 
\eqref{eq:lift} (see Figure \ref{fig:tg-g} (c)).
 
The prolongation operator (or natural inclusion) $\cP_{h}$ can then be defined as
\[
\cP_{h}: \mathcal{V}_H \to \mathcal{V}_h,\; w_H \mapsto \wh{w}_H.
\]
Suppose that the finite element approximation spaces have the following basis 
set
\[
\mathcal{V}_H = \operatorname{span} \{ \phi_{H,i}\}_{i=1}^{N_H},
\text{ and }
\mathcal{V}_h = \operatorname{span} \{ \phi_{h,i}\}_{i=1}^{N_h}.
\]
If $w_H = (W_H)^T \varPhi_H$,  where $\varPhi_H = 
(\phi_{H,1},\dots,\phi_{H,N_H,})^T$, then
\[
\mathcal{V}_h \ni \cP_{h} w^H = \cP_{h} \Big((W_H)^T \varPhi_H\Big) = (P_h W_H)^T 
\varPhi^h,
\]
where $P_h \in \RR^{ N_h\times N_H} $ is the matrix representation of the 
prolongation operator.  Note that the geometric projection is 
implicitly imposed here. 

For example, by \eqref{eq:proj-1}, 
$\cP_{h}\phi_{H,i}(\vz_{H,j}) = \delta_{ij}$, 
and $\cP_{h}\phi_{H,i}(\vz_{h,ij}) = 1/2$ if $i\neq j$, where $\vz_{h,ij}$ is a 
newly created vertex in $\cN_h\backslash \cN_H$ by projecting the midpoint of 
$\vz_{H,i}$ and $\vz_{H,j}$ onto the continuous surface (see Figure 
\ref{fig:tg-g} (c)).

Similarly, the restriction operator $\cP_{H}$ is opted to be the transpose of the 
geometric projection $\cP_{h}$ defined above:
\[
\cP_{H}: \mathcal{V}_h \to \mathcal{V}_H, \; w_h \mapsto \cP_{H} w_h,
\;\text{ and }\;
\mathcal{V}_H\ni \cP_{H} w_h =  (P_H W_h)^T 
\varPhi_H,
\]
where $P_H = (P_h)^T\in \RR^{ N_H\times N_h}$ is the matrix of the 
restriction operator.

\begin{algorithm}[t]
\caption{A two-grid scheme with a spectral shift.}
\label{alg:tg}
\begin{algorithmic}[1]
\State{\textbf{Coarse grid eigensolve}}
\item[] Set a fixed shift \rrevision{$\mu_H :=\mu\geq 0$}, and find $(u_H,\lambda_H)$ satisfying
\begin{equation}
\label{eq:pb-c-eigs}
a_{H} (u_H ,v) -\mu_H  b_H (u_{H}, v)= 
\lambda_H b_H(u_H, v), \quad \text{ for any } v\in \mathcal{V}_H.
\end{equation}

\State {\textbf{Fine grid source approximation}}
\item[] Refine $\cT_H$ and perform the geometric projection to get $\cT_h$, 
construct $\mathcal{V}_h$, approximate $u^h$ by solving the following indefinite source problem on 
$\cT_h$:
\begin{equation}
\label{eq:pb-f-s}
a_{h} (u^h ,v ) - (\mu_H+\lambda_H) b_h(u^h,v) = b_h(\cP_{h} u_H, v), 
\quad \text{ for any } v\in \mathcal{V}_h.
\end{equation}
\end{algorithmic}
\end{algorithm}

Given these definitions, the 
two-grid method approximating the exact solution of problem 
\eqref{eq:pb-eig-fea} is given by Algorithm \ref{alg:tg}.

\begin{remark}[Natural extension to a multilevel method]
\label{rmk:tg-cascade}
When multiple 
levels of meshes are available ($V_{h_k}$ for $k = 1,\dots,K$ with $K\geq 3$), 
Algorithm \ref{alg:tg} can be naturally extended to 
be a multilevel method by being applied in a cascading fashion between two adjacent 
levels. For example, starting from $V_{h_1}$, when a two-grid eigenpair 
approximation $(u^{h_2},\lambda^{h_2})$ is obtained, we set 
$(u_{h_2},\lambda_{h_2})\gets (u^{h_2},\lambda^{h_2})$. Then step 2 and step 3 in 
Algorithm \ref{alg:tg} are repeated for level 3 to level $K$.
\end{remark}
\medskip

\begin{remark}[Approximation accuracy of the source problem]
In Algorithm \ref{alg:tg}, the source problem \eqref{eq:pb-f-s} can be approximated 
by a direct or multilevel method. We note that if a multilevel hierarchy exists and 
the two-grid method is applied in a cascading fashion, then numerically the 
source problem on $\mathcal V_{h_k}$ of the $k$-th level ($k\geq 3$) 
does not necessarily require to be solved at an accuracy of a direct solver. This implies that a smoother 
(relaxation method) can be applied to problem \eqref{eq:pb-f-s} in Step 2 of 
Algorithm \ref{alg:tg}, with the approximation from the previous level as initial 
guess, instead of a direct solve. Below, in the Algorithms 
\ref{alg:mg-correction} and \ref{alg:mg-vcycle}, when the term ``approximate'' is 
used for the source problems, the user can choose a direct/multigrid solver 
or smoother. We illustrate this numerically below in Section 
\ref{sec:n-rs}.
\end{remark}

\subsection{Convergence analysis}

The main ingredient in proving the convergence of Algorithm \ref{alg:tg} 
is to bridge the connection 
between the geometric projection on the surface \cite{Dziuk2013finite} with 
existing two-grid convergence results (e.g., see \cite{HuCheng11ieig,ZhouHul14MaxEig}) for the Laplace eigenvalue problem on the plane. 
The error introduced by the projection between non-hierarchical spaces that arises 
in this setting is accounted for in the final estimate we 
derive.

We use the following lemmas to obtain the convergence estimate of the 
two-grid method for the Laplace-Beltrami eigenvalue problem.  The first lemma, Lemma 
\ref{lem:infsup} (e.g. see \cite{babuvska1989finite,HuCheng11ieig}), gives the 
stability estimate for the discrete shifted problem.

\begin{lemma}[Discrete inf-sup condition with a shift]
\label{lem:infsup}
If $\mu$ is not an exact eigenvalue to problem \eqref{eq:pb-eig-fea}, then there 
exists a constant $C(\mu)$ such that
\[
\sup_{v\in \mathcal{V}_h} \frac{\abs{a_h(u_h,v) - \mu b_h(u_h,v) }}{\snorm{v}_{H^1(\G_h)}} 
\geq C(\mu) \snorm{u_h}_{H^1(\G_h)}, \quad \forall u_h\in \mathcal{V}_h
\]
\end{lemma}
The next lemma, Lemma \ref{lem:eigenid}, is an important identity used to prove 
the rate of convergence for the approximation of a certain eigenvalue (e.g., see 
\cite{babuvska1989finite}).
\begin{lemma}
\label{lem:eigenid}
Let $(u,\lambda) \in H^1(\G)\times \RR^+ $ be an eigenpair for problem 
\eqref{eq:pb-eig-weak}, then for any $w\in H^1(\G)\backslash\{0\}$
\begin{equation}
\frac{a( w, w)}{b( w, w)} - \lambda = \frac{a( w- u, w- u)}{b( w, w)} -
\lambda \frac{b( w- u, w- u)}{b( w, w)}.
\end{equation}
\end{lemma}

We then define the following measure of $\mathcal{V}_h$'s 
approximation power,
\begin{equation}
\label{eq:eta}
\eta(h) :=  \sup_{f\in H^1(\G_h), \snorm{f}_{H^1} = 1} \inf_{v\in \mathcal{V}_h}
\snorm{Tf - v}_{H^1(\G_h)},
\end{equation}
where the operator $T$ is defined as:
\begin{equation}
a_h(Tf,v) = b_h(f,v), 
\quad \forall f\in H^1(\G_h),\;\text{ and }\;\forall \,v\in H^1(\G_h).
\end{equation}

For finite element eigenvalue problems on 
a convex planar domain $\Omega\subset \RR^2$, one can assume that the eigenspace 
$M(\lambda)$ has certain regularity, e.g., one assumes that 
$M(\lambda)\subset H^2(\Omega)$, which is the Sobolev space containing functions 
with second weak derivatives that are $L^2$-integrable. As a 
result, 
it is shown in \cite{HuCheng11ieig} that, if hierarchical coarse and fine finite 
element spaces are used, i.e., $V_{H} \subset \mathcal{V}_h$, the general estimate 
for the two-grid approximation reads for the linear finite element approximation:  
\begin{equation}
\label{eq:est-tg}
\min_{\alpha\in \RR} \norm{u - \alpha u^h}_{H^1(\G_h)} \leq 
C(\lambda) \left(h + H^4\right),
\text{ and }\quad |\lambda - \lambda^h|\leq 
C(\lambda) \left( h^2 + H^8\right).
\end{equation}
However, in the case of surface finite element here, Theorem 3.3 in 
\cite{HuCheng11ieig} cannot be directly applied due to the facts that (i) only 
$H^1(\G_h)$, not $H^2(\G_h)$ is well-defined on a piecewise linear triangulation 
$\G_h$ (see e.g., \cite[Section 2]{Dziuk1988finite}), where $H^2(\G_h)$ stands for, 
when being well-defined:
\begin{equation}
\label{eq:sp-h2}
H^2(\G_h) : = \big\{ \bar v\in H^1(\G_h): (\nabla_{\G} \bar v)_i \in H^1(\G_h), \; 
i=1,2,3 \big\};
\end{equation}
(ii) when refining, the spaces are not hierarchical, i.e., $V_{H} 
\not\subset \mathcal{V}_h$ (see Figure 
\ref{fig:tg-g}).

In the rest of this subsection, a modified two-grid convergence proof 
is presented following the method from \cite{ZhouHul14MaxEig}, and similar bounds 
are obtained as the standard two-grid results in \eqref{eq:est-tg} given 
that the geometric error is of $O(h^2)$.

\begin{theorem}[Convergence of the two-grid method]
\label{th:tg}
Let $\lambda_H$ be an approximation to the eigenvalue $\lambda$ of problem 
\eqref{eq:pb-eig-weak} satisfying the a priori estimate in Lemma \ref{lem:apriori}. 
\rrevision{Consider an approximation $u^h$ obtained from the two-grid method 
given in Algorithm \ref{alg:tg} with $\mu = 0$ on $\G_h$ with $h$ sufficiently 
small, and $\lambda^h:= a_h(u^h,u^h)/b_h(u^h,u^h)$.} Then, there exists an eigenfunction $u\in M(\lambda)$ such that the following estimates 
hold
\begin{equation}
\label{eq:est-tgs}
\min_{\a\in \RR}\left\vert{u - \a\wt{u}^h}\right\vert_{H^1(\G)} 
\leq C(\lambda) \bigl(h + H^4 \bigr),
\quad
\text{and} \quad |\lambda - \lambda^h|\leq 
C(\lambda) \bigl( h^2 + H^8\bigr).
\end{equation} 
\end{theorem}

\begin{proof}
The proof follows from the ones in \cite{HuCheng11ieig,ZhouHul14MaxEig}. Assume that the 
coarse approximation $\lambda_H$ is not an eigenvalue of the 
discrete eigenvalue problem \eqref{eq:pb-eig-fea} on the fine mesh. 
Consider an auxiliary solution 
$\wh{u}^h = (\lambda_h - \lambda_H)u^h$, where $u^h$ solves 
problem \eqref{eq:pb-f-s} with $\mu = 0$, then it can be verified that this 
$\wh{u}^h$ satisfies:
\[
a_{h} (\wh{u}^h ,v ) - \lambda_H b_h(\wh{u}^h,v)
= (\lambda_h-\lambda_H) b_h(\cP_h {u}_H, v), 
\quad \forall\, v\in \mathcal{V}_h. 
\]
In the equation above, $(u_h,\lambda_h)$ is an eigenpair obtained from the direct solve for problem \eqref{eq:pb-eig-fea}, i.e., $a_h(u_h,v) 
= \lambda_h b_h(u_h,v)$ for any $v \in \mathcal{V}_h$. Taking the 
difference of these two equations yields: for any $v \in \mathcal{V}_h$ 
\begin{equation}
a_h(u_h - \wh{u}^h,v) - \lambda_H b_h(u_h- \wh{u}^h, v)
= (\lambda_h - \lambda_H) b_h(u_h - \cP_h u_H,v).
\end{equation}
Applying the discrete inf-sup stability estimate in Lemma \ref{lem:infsup}, we have
\begin{equation}
\label{eq:err-u}
\begin{aligned}
C \snorm{u_h - \wh{u}^h}_{H^1(\G_h)} 
&\leq \sup_{v\in \mathcal{V}_h} \frac{\abs{a_h(u_h - \wh{u}^h,v) - \lambda_H b_h(u_h - 
\wh{u}^h,v) }}{\snorm{v}_{H^1(\G_h)}} 
\\
&= \sup_{v\in \mathcal{V}_h} \frac{\abs{(\lambda_h - \lambda_H) b_h(u_h - \cP_h 
u_H,v)}}{\snorm{v}_{H^1(\G_h)}} .
\end{aligned}
\end{equation}
By the triangle inequality and the a priori estimate from Lemma \ref{lem:apriori}, 
we have
\begin{equation}
\abs{\lambda_h - \lambda_H} \leq \abs{\lambda - \lambda_h} + \abs{\lambda - 
\lambda_H} \leq C(\lambda) H^2.
\end{equation}
Now we restrict the true eigenfunction $u\in M(\lambda)$ with its continuous mapping 
$\bar{u}$ to the discrete surface $\G_h$ or $\G_H$, and use the definition of $\eta(\cdot)$, the direct solve achieving the best approximation (Lemma \ref{lem:apriori}), together with the 
geometric error (Lemma \ref{lem:lift}):
\begin{equation}
\begin{aligned}
& \sup_{v\in \mathcal{V}_h} \frac{b_h(u_h - \cP_h u_H,v)}{\snorm{v}_{H^1(\G_h)}}
\leq  \sup_{v\in H^1(\G_h)} \frac{b_h(u_h - \bar{u},v)}{\snorm{v}_{H^1(\G_h)}}
+ \sup_{v\in H^1(\G_h)} \frac{b_h(\bar{u} - \cP_h u_H,v)}{\snorm{v}_{H^1(\G_h)}}
\\ \leq \;& c_1(\lambda) \eta(h) \snorm{u_h - \bar{u}}_{H^1(\G_h)} + 
c_2(\lambda) \eta(H) \snorm{\bar{u} - u_H}_{H^1(\G_H)} +c_3 \eta(H) H^2.
\end{aligned}
\end{equation}
Now for $\eta(h)$ and $\eta(H)$ define in \eqref{eq:eta}, since restricting on each 
element $T$, the eigenfunction is smooth, then by the interpolation estimate in  
\cite[Lemma 2.2]{Demlow2007adaptive}, we have
\begin{equation}
\eta(h) = \sup_{f\in H^1(\G_h), \snorm{f}_{H^1} = 1} 
\inf_{v\in \mathcal{V}_h}\snorm{Tf - v}_{H^1(\G_h)} \leq C h,
\end{equation}
where $C$ only depends on the geometry. Using this result, \eqref{eq:err-u}, and the 
a priori estimate for the direct solve eigenfunction approximation, we have:
\begin{equation}
\snorm{u_h - \wh{u}^h}_{H^1(\G_h)} \leq C(\lambda)H^4.
\end{equation}
Then apply the geometric error estimates \eqref{eq:err-geom} for both $\wh{u}^h$ and 
$u_h$,
\begin{equation}
\big| u - \wt{\wh{u}^h} \big|_{H^1(\G)} \leq \snorm{u - \wt{u}_h}_{H^1(\G)} + 
c(\lambda)h^2 + \snorm{u_h - \wh{u}^h}_{H^1(\G_h)} \leq C(\lambda) (h + H^4),
\end{equation}
where $h$ is assumed to be small enough such that the geometric error, which is 
$O(h^2)$, can be omitted compared with the $O(h)$ term. 

Lastly, using the fact that $\wh{u}^h = (\lambda_h - \lambda_H)u^h$ we have
\begin{equation}
\lambda^h = \frac{a_h(u^h,u^h)}{b_h(u^h,u^h)} = 
\frac{a_h(\wh{u}^h,\wh{u}^h)}{b_h(\wh{u}^h,\wh{u}^h)}.
\end{equation}
By Lemma \ref{lem:eigenid}, the eigenfunction to be 
normalized to satisfy $b_h(u^h,u^h)=1$, Poincar\'{e} inequality, and Lemma \ref{lem:lift}, we can get the estimate for the two-grid approximation:
\begin{equation}
\label{eq:est-lambdah}
\begin{aligned}
|\lambda^h - \lambda| &=  \left\vert 
\frac{a_h(u^h,u^h)}{b_h(u^h,u^h)} -\lambda\right\vert
\\
&\leq  \left\vert 
a(\wt{u}^h,\wt{u}^h) -\lambda\right\vert + 
 \left\vert {a_h(u^h,u^h)} -a(\wt{u}^h,\wt{u}^h) \right\vert
\\
& \leq  \left\vert a(\wt{u}^h - u,\wt{u}^h - u) 
-\lambda b(\wt{u}^h - u,\wt{u}^h - u)\right\vert + Ch^2
\\
& \leq \snorm{\wt{u}^h - u}_{H^1(\G)}^2 + \lambda \norm{\wt{u}^h - u}_{L^2(\G)}^2 + Ch^2
\\
& \leq C(\lambda) \snorm{\wt{u}^h - u}_{H^1(\G)}^2 +
Ch^2 \leq C(\lambda) \bigl( h^2 + H^8\bigr).
\end{aligned}
\end{equation}
\end{proof}

\begin{remark}
\label{rmk:tg}
Theorem \ref{th:tg} implies that if the mesh sizes are chosen such that 
$H\lesssim h^{1/4}$ between neighboring levels, then 
the optimal linear rate of convergence for the eigenfunction in 
$\snorm{\cdot}_{H^1(\G)}$ and the quadratic 
convergence for the eigenvalue follow. In our setting, assuming multiple levels of 
meshes (obtained by uniformly refining the mesh from the previous level) and that we 
project the vertices onto the surface, $H\lesssim h^{1/4}$ holds and the estimate 
follows. Assume Algorithm \ref{alg:tg} is applied in a cascading fashion spanning 
multiple levels, then the optimal convergence rates of these two algorithms depend 
on the assumption that the coarse mesh is fine enough, that is, they depend on the assumption that the geometric error is sufficiently small. 
\end{remark}

\section{The Bootstrap Multigrid Method}
In this section, we propose a finite element bootstrap multigrid (BMG) eigensolver 
based on the bootstrap algebraic multigrid (BAMG) framework~\cite{BAMG2010}. The 
key ingredient in the bootstrap approach proposed in this section is the idea 
to continuously enrich the coarse space with computed eigenfunction approximations 
coming from linear solves on finer meshes.

The motivation of the eigensolve in an enriched coarse 
space is to overcome the drawbacks of the standard two-grid method. Essentially, 
the two-grid methods proposed, when using the original geometrically defined coarse 
space, accelerate the eigensolve on the finest mesh.  However, the number of 
correctly approximated eigenpairs by these two-grid method depend on the dimension 
of the coarse space (see the numerical example in Section 
\ref{sec:n-tg}). With the BMG eigensolver, the desired eigenpairs of the 
original Laplace-Beltrami operator can be approximated with same order of accuracy 
as the direct eigensolve achieves on the finest mesh, assuming certain mesh size 
relations are satisfied between consecutive levels (see Remark \ref{rmk:tg}).

\begin{algorithm}[ht]
\caption{A two-grid BMG scheme for approximating eigenfunctions near $\mu$ using an 
enriched coarse space.}
\label{alg:mg-correction}
\begin{algorithmic}[1]
\State{\textbf{Coarse grid eigensolve}}
\item[] Set a coarse grid shift $\mu_{H} = \mu\geq 0$, find 
$(u_{H,i},\lambda_{H,i})\in \mathcal{V}_H\times 
\RR^+$ (for $i = 1,\dots,N$, using direct eigensolve, 
where $N\leq \dim \mathcal{V}_H$) 
satisfying
\begin{equation}
a_H (u_{H,i} ,v) - \mu_{H}  b_H (u_{H,i}, v)= \lambda_{H,i} b_H (u_{H,i}, v), 
\;\; \text{ for any } v\in \mathcal{V}_H.
\end{equation}

\State{\textbf{Choose eigenfunctions for the enrichment}}
\item[] Let the index set for the enrichment candidate eigenfunctions be 
$\Lambda \subset \{1,\dots,N \}$, where if $i\in \Lambda$, 
$|\lambda_{H,i} - \mu| <\texttt{tol}$. Let
\begin{equation}
\label{eq:sp-e}
X_H := \operatorname{span} \{u_{H,i}\}_{i\in \Lambda}.
\end{equation}

\State{\textbf{Fine grid source approximation} }
\item[] Refine $\cT_H$ and perform the geometric projection to get $\cT_h$. 
Approximate $ u^{h,i}\in \mathcal{V}_h$, 
where $i\in \Lambda$, using $u_{H,i} \in X_H$ as the source, in 
\begin{equation}
\label{eq:pb-f-aux}
a_h (u^{h,i} ,v) - \mu_{H}  b_h (u^{h,i}, v)= 
\lambda_{H,i} b_h (\cP_{h} u_{H,i}, v), 
\;\; \text{ for any } v\in \mathcal{V}_h.
\end{equation}
Then orthogonalize ${u}^{h,i}$'s with respect to the inner product 
$b_h(\cdot,\cdot)$, let the enrichment space contain the orthogonalized source 
approximations:
\begin{equation}
X_h := \operatorname{span} \{{u}^{h,i}\}_{i\in \Lambda}.
\end{equation}

\State{\textbf{Coarse grid eigensolve in the enriched space}}
\item[] Set a new shift $\mu_{h}$. Find 
$(u_{h,i},\lambda_{h,i})\in V_{H,h}\times \RR^+$
satisfying, for $i\in \Lambda$:
\begin{equation}
\label{eq:pb-c-aux}
a_{H,h} (u_{h,i} ,v) - \mu_{h} b_{H,h}(u_{h,i}, v) = 
\lambda_{h,i} b_{H,h} (u_{h,i}, v), 
\;\; \text{ for any } v\in V_{H,h},
\end{equation}
where $V_{H,h} := \mathcal{V}_H + X_h$ is the enriched coarse space. Update 
$X_H = \operatorname{span} \{u_{h,i}\}_{i\in \Lambda}$.
\end{algorithmic}
\end{algorithm}

\subsection{The two-grid bootstrap algorithm} 
A two-grid bootstrap algorithm is outlined in Algorithm \ref{alg:mg-correction} and 
illustrated briefly in Figure~\ref{fig:bmg-t}. 
The algorithm takes  
a coarse mesh, a shift 
 $\mu$, and the tolerance $\texttt{tol}$ as inputs.

In the two-grid bootstrap method, the bilinear forms $a_{H,h}(\cdot,\cdot)$ and 
$b_{H,h}(\cdot,\cdot)$ are defined as follows.  Let $w \in \mathcal{V}_H + 
X_h$ be any function such that $w = w_H + w_h$, where $w_H\in \mathcal{V}_H$, and $w_h \in 
X_h$. Then, for any test function $v = v_H + v_h \in \mathcal{V}_H + X_h$ 
\begin{equation}
\begin{aligned}
a_{H,h} (w ,v) := a_H(w_H,v_H) + a_h(\cP_{h} w_H, v_h)
\\ 
+ a_h(w_h, \cP_{h} v_H) + a_h(w_h,v_h),
\\
\text{and }\quad 
b_{H,h} (w ,v) := b_H(w_H,v_H) + b_h(\cP_{h} w_H, v_h) 
\\
+ b_h(w_h, \cP_{h} v_H) + b_h(w_h,v_h).
\end{aligned}
\end{equation}

Now we rewrite some of the key equations in Algorithm \ref{alg:mg-correction} in 
their matrix forms. Problem \eqref{eq:pb-f-aux} can be 
rewritten as: find $U^{h,i} \in \RR^{N_h}$, $i\in \Lambda$
\begin{equation}
\label{eq:pb-f-aux-mat}
( A_h - \mu_{H} M_h) U^{h,i} = \lambda_H M_h (P_h U_{H,i}).
\end{equation}
Here $A_h$ and $M_h$ are the stiffness matrix and mass matrix for the degrees of 
freedom on the fine approximation space $\mathcal{V}_h$, respectively. $U^{h,i}$ is 
the vector representation of $u^{h,i}$ in the 
canonical finite element basis, and its superscript is inherited from $u^{h,i}$. The 
$U_{H,i}$ with the subscript is the vector representation of the direct solve 
solution $u_{H,i}$ in the coarse approximation space $\mathcal{V}_H$. 

Step 4 of Algorithm \ref{alg:mg-correction} is to find $U_{h,i} 
\in \RR^{N_H+|\Lambda|}$, and $\lambda_{h,i} \in \RR^+$, $i\in \Lambda$
\begin{equation}
\label{eq:pb-c-aux-mat}
( A_{H,h} - \mu_{h} M_{H,h})U_{h,i} = \lambda_{h,i} M_{H,h} U_{h,i}.
\end{equation}
The enriched stiffness and mass matrices $A_{H,h}$ and $M_{H,h}$ are in the 
following coarse-fine block form: let ${U}^h := (U^{h,1},\dots, 
U^{h,|\Lambda|}) \in 
\RR^{N_h\times |\Lambda|}$ be the block of all the approximations from 
problem \eqref{eq:pb-f-aux-mat}, and
\[
A_{H,h} = \begin{pmatrix}
A_H & P_H A_h {U}^h 
\\[3mm]
({U}^h)^T A_h P_h & ({U}^h)^T A_h {U}^h
\end{pmatrix},\;
M_{H,h} = \begin{pmatrix}
M_H & P_H M_h {U}^h
\\[3mm]
({U}^h)^T M_h P_h & ({U}^h)^T M_h {U}^h
\end{pmatrix}.
\]

\begin{remark}[Choice of the shifts]
The shifts $\mu_{H}, \mu_{h}\geq 0$ in Algorithm \ref{alg:mg-correction} are added 
in case the user is interested in a 
specific range of the eigenvalues. If one is to find the eigenvalues from the 
smallest one, the shift can be set as $\mu_{H}= \mu_{h} = 0$ for all of the enriched 
coarse eigenvalue problems, and fine source approximation problems. To 
recover interior eigenvalues, \revision{the coarse grid shift $\mu_H$ can be set to be a 
positive number that is near the center of an eigen-cluster of interest. Then the new shift $\mu_h$ is updated using the Rayleigh quotients computed from the fine source approximations.}
\end{remark}

The choice of the set of 
enrichment functions $X_H$ (test 
vectors in the BAMG (\cite{BAMG2010}) setup) with index set $\Lambda$ is related to the eigenspace of 
interest, as well as the resources one can afford. 
As the dimension of $X_H$ increases, the sparsity of $A_{H,h}$ deteriorates resulting a more expensive the fine grid source problem. Thus, we propose that the user shall choose the coarse space accordingly to the spectrum of interest. 
Using approximating the Laplace-Beltrami 
spectrum on a closed surface as an example, 
consider the simple case that the user wants to recover the $l$-th 
eigenpair, $(\lambda_{h,l},u_{h,l})$, where $1\leq l< \dim \mathcal{V}_H$. The 
eigensolve in the coarse space contains the discrete approximations 
$\{(\lambda_{H,i},u_{H,i})\}_{i=1}^{\dim \mathcal{V}_H}$ to these
eigenpairs. Then, the index set for the eigenfunction approximations
$\Lambda$ in \eqref{eq:sp-e} can be chosen as $\{\hat{k}\in \ZZ: l - m\leq \hat{k} 
\leq l\}$.  Here, $m$ is greater than or equal to the geometric multiplicity shown in 
the discrete spectra for the eigenvalue closest to the eigenvalue of interest, say 
$\lambda_{H,l}$, and the tolerance \texttt{tol} is set to be 
the maximum distance between the distinctive eigenvalue clusters. Intuitively, these
choices are motivated by the fact that the 
algorithm should ``detect'' the improvement in the approximations of
the eigenpair $(\lambda_{h,l},u_{h,l})$ of interest. For 
additional discussion of how to set shifts and how to choose the enrichment 
candidates, please refer to the examples in Section \ref{sec:n-bm}.

\begin{theorem}[Error estimates for BMG Algorithm~\ref{alg:mg-correction}]
\label{th:btg}
Assume $\mu\geq 0$ is neither an eigenvalue of problem \eqref{eq:pb-eig-fea} or problem \eqref{eq:pb-eig-dis}, and suppose $u_{H,i}\in X_H$ in \eqref{eq:sp-e} for $i=1,\dots, \dim X_H$ are linearly independent, and there exists a $u_i\in M({\lambda+\mu})$ with the estimates:
\begin{equation}
\label{eq:asp-bootstrap}
\begin{gathered}
|\lambda - \lambda_{H,i}|\leq 
C(\lambda,\mu) H^2, \quad \text{ and }\quad 
\bigl\vert{u_i - \wt{u}_{H,i} }\bigr\vert_{H^1(\G)}  \leq C(\lambda,\mu) H,
\\
\sup_{v\in \mathcal{V}_h/\RR} \frac{b(u_i - 
\wt{\cP_{h} u}_{H,i},\wt{v})}{\snorm{v}_{H^1(\G_h)}}
\leq 
C(\lambda,\mu) H \bigl\vert{u_i - \wt{u}_{H,i} }\bigr\vert_{H^1(\G)}.
\end{gathered}
\end{equation}
After one iteration of Algorithm \ref{alg:mg-correction} with $\mu_H = \mu_h= \mu$ 
under a uniform refinement with $h$ sufficiently small, and an 
exact solve in Step 3, there  
exists a $\wh{u}\in M({\lambda+\mu})$ , such that the 
resultant approximation $(\lambda_{h,i}, u_{h,i})$ has the following error estimates:
\begin{equation}
\label{eq:est-bootstrap}
|\lambda - \lambda_{h,i}|\leq 
C(\lambda, \mu) (h^2 + H^4), \quad \text{ and }\quad 
\bigl\vert{\wh{u} - \wt{u}_{h,i}}\bigr\vert_{H^1(\G)}  \leq C(\lambda, \mu) ( h+H^2).
\end{equation}
\end{theorem}

\begin{proof}
The proof follows from a similar argument in \cite[Theorem 3.4]{Lin2015multi} and uses Lemma \ref{lem:infsup}, with the geometric error estimates from Lemma \ref{lem:lift}. For simplicity first we
consider $\dim X_H = 1$, define a projection $\Pi_h : H^1(\G) \to \mathcal{V}_h$ of a 
$u\in M(\lambda+\mu)$ with respect to the inner product $a_h(\cdot, \cdot)$ as 
follows: 
\begin{equation}
\label{eq:proj-s}
a_h\big(\Pi_h u, v\big) = a(u,\wt{v}), \quad \text{and } \int_{\G_h}\Pi_h u \,dS = \int_{\G} u \,dS = 0, \quad \forall v\in \mathcal{V}_h.
\end{equation}
For $u^h$ from the exact solve of \eqref{eq:pb-f-aux}, by Lemma \ref{lem:infsup}, 
\begin{equation}
c \snorm{u^h - \Pi_h u}_{H^1(\G_h)} \leq \sup_{v\in \mathcal{V}_h} 
\frac{\abs{a_h(u^h - \Pi_h u,v) - \mu b_h(u^h - \Pi_h u,v) }}{\snorm{v}_{H^1(\G_h)}} .
\end{equation}
The numerator on the right hand side above becomes
\[
\begin{aligned}
& \abs{a_h(u^h - \Pi_h u,v) - \mu b_h(u^h - \Pi_h u,v) }
\\
= & \,\big| a_h(u^h,v) - \mu b_h(u^h,v) 
- (\underbrace{a_h(\Pi_h u,v) - \mu b_h(\Pi_h u ,v)}_{=:(\mathfrak{d})}) \big|.
\end{aligned}
\]
The first difference $a_h(u^h,v) - \mu b_h(u^h,v) = \lambda_{H} b_h (\cP_{h} u_{H}, v)$. For $(\mathfrak{d})$ by \eqref{eq:proj-s} and $a(u,\wt{v}) -\mu b(u,\wt{v}) = \lambda b(u,\wt{v})$ we can write it as
\[
\begin{aligned} 
& (\mathfrak{d}) =a(u,\wt{v}) -\mu b_h(\Pi_h u ,v) 
\\
= & \; \lambda b(u, \wt{v}) + \mu b(u, \wt{v}) - \mu b(\wt{\Pi_h u}, \wt{v}) 
+ \mu b(\wt{\Pi_h u}, \wt{v})- \mu b_h(\Pi_h u ,v). 
\end{aligned}
\]
As a result, by triangle inequality we have:
\begin{equation}
\label{eq:pb-tg-infsup}
\begin{aligned}
& \abs{a_h(u^h - \Pi_h u,v) - \mu b_h(u^h - \Pi_h u,v) }   
\\
\leq & \, \underbrace{\abs{ \lambda_{H} b_h (\cP_{h} u_{H}, v) - \lambda b(u, \wt{v})}}_{=:(\rm I)}
\\
& \, +  \underbrace{\bigl| \mu b(u, \wt{v}) - \mu b(\wt{\Pi_h u}, \wt{v})\bigr|}_{=:(\rm I\!I)}
+ \underbrace{\bigl| \mu b(\wt{\Pi_h u}, \wt{v})- \mu b_h(\Pi_h u ,v)\bigr|}_{=:(\rm I\!I\!I)}.
\end{aligned}
\end{equation}
Next we show the estimates for each term on the right of \eqref{eq:pb-tg-infsup}. 
For $(\mathrm{I})$ we have:
\begin{equation}
\label{eq:pb-tg-pf}
\begin{aligned}
& (\mathrm{I})= \big|\lambda_{H} b_h (\cP_{h} u_H, v) -  \lambda b(u,\wt{v}) \big|
\\
\leq 
&\abs{  b\big(\lambda_H \wt{\cP_{h} u}_H - \lambda u,\wt{v}\big)}
+ \abs{b_h\big(\lambda_H \cP_{h} u_H ,v\big) - 
b\big(\lambda_H \wt{\cP_{h} u}_H,\wt{v}\big)}.
\end{aligned}
\end{equation}
For the second term above in \eqref{eq:pb-tg-pf} which is introduced by approximating the geometry, by letting $q = u^h - \Pi_h u \in \mathcal{V}_h$ below, applying \cite[Lemma 4.7]{Dziuk2013finite}, a Poincar\'{e} inequality on $\Gamma_h$ since $\int_{\G_h}(u^h-\Pi_h u) \,dS=0$ (\cite[Remark 5.3]{reusken2014analysis}), and 
discarding higher order terms, we have
\begin{equation}
\label{eq:pb-tg-s}
\begin{aligned}
&\abs{b_h\big(\lambda_H \cP_{h} u_H ,v\big) - 
b\big(\lambda_H \wt{\cP_{h} u}_H,\wt{v}\big)}
\\
\leq & \, \lambda_H 
\sup_{q\in \mathcal{V}_h/\RR}
\frac{\big| b_h\big(\lambda_H \cP_{h} u_H,q\big) - 
b\big(\lambda_H \wt{\cP_{h} u}_H,\wt{q}\big) \big|}{|q|_{H^1(\Gamma_h)}}
\snorm{v}_{H^1(\Gamma_h)}
\\
\leq & \, ch^2 \lambda_H  \frac{\norm{\cP_{h} u_H}_{L^2(\Gamma_h)} 
\norm{q}_{L^2(\Gamma_h) }}{|q|_{H^1(\Gamma_h)}} \snorm{v}_{H^1(\Gamma_h)}
\leq ch^2 \lambda_H \norm{\cP_{h} u_H}_{L^2(\Gamma_h)}
\snorm{v}_{H^1(\Gamma_h)}.
\end{aligned}
\end{equation}
For the first term in \eqref{eq:pb-tg-pf}, applying the geometric error estimate in 
Lemma \ref{lem:lift} such that $\snorm{\wt{v}}_{H^1(\Gamma)} \leq 
(1+ch^2)\snorm{v}_{H^1(\Gamma_h)}$, we 
have:
\begin{equation}
\begin{aligned}
& b\big(\lambda_H \wt{\cP_{h} u}_H - \lambda u,\wt{v}\big)
\leq (1+ch^2)\sup_{q\in \mathcal{V}_h/\RR}
\frac{b\big(\lambda_H \wt{\cP_{h} u}_H - \lambda u,\wt{q}\big)}{|q|_{H^1(\Gamma_h)}}
\snorm{v}_{H^1(\Gamma_h)}
\\
\leq &(1+ch^2) \left( |\lambda-\lambda_H|
\sup_{q\in \mathcal{V}_h/\RR}
\frac{b\big(\wt{\cP_{h} u}_H,\wt{q}\big)}{|q|_{H^1(\Gamma_h)}}
 + \lambda \sup_{q\in \mathcal{V}_h/\RR}
\frac{b\big(\wt{\cP_{h} u}_H - u,\wt{q}\big)}{|q|_{H^1(\Gamma_h)}} \right) 
 \snorm{v}_{H^1(\Gamma_h)}.
\end{aligned}
\end{equation}
Exploiting the technique in \eqref{eq:pb-tg-s} again, and combining estimates above with assumption in \eqref{eq:est-bootstrap} we have reached:
$(\mathrm{I}) \leq C(1+ch^2)^2 H^2\,  \snorm{v}_{H^1(\Gamma_h)}$.
For $(\mathrm{I\!I})$, we have 
\[
(\mathrm{I\!I}) := \bigl| \mu b(u, \wt{v}) - \mu b(\wt{\Pi_h u}, \wt{v})\bigr|
\leq \mu \big\|u - \wt{\Pi_h u} \big\|_{L^2(\Gamma_h)} \norm{v}_{L^2(\Gamma_h)}.
\]
We choose $v = u^h - \Pi_h u \in \mathcal{V}_h$ here thus Poincar\'{e} inequality can be used and 
$ \norm{v}_{L^2(\Gamma_h)} \lesssim \snorm{v}_{H^1(\Gamma_h)}$.
Since the lifting operator does not preserve the integration, Poincar\'{e} inequality (see e.g., \cite[Lemma B.63]{Ern;Guermond:2013Theory}) has to be applied on the first term as follows,
\[
c\big\|u - \wt{\Pi_h u} \big\|_{L^2(\Gamma_h)} \leq \big|u - \wt{\Pi_h u}\big|_{H^1(\Gamma)}
+ |\Gamma|^{-1} \int_{\Gamma} \wt{\Pi_h u}\,dS.
\]
A standard error estimate for the projection $\big|u - \wt{\Pi_h 
u}\big|_{H^1(\G)} \lesssim h |\nabla_{\G}u|_{H^1(\Gamma)}$ (e.g., \cite[Lemma 
2.2]{Demlow2007adaptive}) can be applied to the first term above. The second term above can be estimated by the fact that $\int_{\Gamma_h} \Pi_h u \,dS = 0$,
\[
\int_{\Gamma} \wt{\Pi_h u}\,dS = b(\wt{\Pi_h u}, 1) - b_h({\Pi_h u}, 1) \leq c h^2 |\Gamma|^{1/2} 
\|{\Pi_h u} \|_{L^2(\Gamma_h)}.
\]
As a result, we have $(\mathrm{I\!I}) \leq C\mu(h + h^2) \snorm{v}_{H^1(\Gamma_h)}$. Lastly for 
$(\rm I\!I\!I)$ \cite[Lemma 4.7]{Dziuk2013finite} is applied again as in \eqref{eq:pb-tg-s} we have 
$(\mathrm{I\!I\!I})\leq C\mu h^2 \snorm{v}_{H^1(\Gamma_h)} $.
Applying the triangle inequality, and discarding the higher order term yield
\[
\big|{u - \wt{u^h}\big|}_{H^1(\G)}
\leq \big|{u - \wt{\Pi_h u}\big|}_{H^1(\G)} + \big|{\wt{\Pi_h u} - \wt{u^h}\big|}_{H^1(\G)} \leq C(\lambda,\mu) (h+H^2).
\]
When $\dim X_H >1$, the argument above still applies by the argument above \cite[Theorem 8.1]{Boffi2010finite} since $u^{h,i}$'s are orthogonal with respect to $b_h(\cdot, \cdot)$.
Lastly, we apply the eigenpair approximation results with multiplicity $\geq 1$ 
results in \cite[Page 700 (8.47b) and Theorem 9.1]{babuska-osborn}, for each $i$, we 
have for a $\wh{u}\in M({\lambda})$, $\mathcal{V}_{H,h}:= 
\mathcal{V}_H+\operatorname{span} \{{u}^{h,i}\}$
\[
\bigl\vert{\wh{u} - \wt{u}_{h,i} } \bigr\vert_{H^1(\G)}
\leq \inf_{w\in M(\lambda)} \inf_{v\in \mathcal{V}_{H,h}}
\bigl\vert{w- \wt{v} } \bigr\vert_{H^1(\G)}\leq 
\big|{u_i - \wt{u^{h,i}}\big|}_{H^1(\G)}, 
\]
which, together with the estimate above, gives the first estimate in 
\eqref{eq:asp-bootstrap}. Now using the same 
argument as \eqref{eq:est-lambdah} yields the eigenvalue estimate in \eqref{eq:est-bootstrap}.
\end{proof}

\begin{remark}[Smoothers versus solve in \eqref{eq:pb-f-aux}]
For the simplicity of the presentation, we have assumed that \eqref{eq:pb-f-aux} is solved exactly in the proof of Theorem \ref{th:btg}. \rrevision{Nevertheless, applying similar estimates for the Gauss-Seidel smoother using, e.g., \cite[Lemma 3.9]{Urschel;Hu;Xu;Zikatanov:2014cascadic}, the estimate in \eqref{eq:est-bootstrap} can be shown to yield the same order error under appropriate choice of smoothers.}
\end{remark}

\begin{figure}[htbp]
  \centering
\begin{subfigure}[b]{0.4\linewidth}
    \centering
    \includegraphics[width=0.9\textwidth]{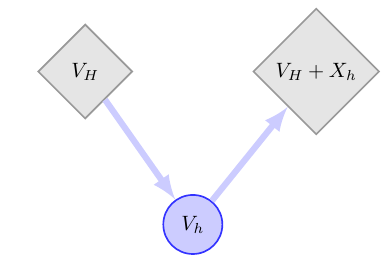}
    \caption{\label{fig:bmg-t}}
\end{subfigure}%
\hspace{0.3in}
\begin{subfigure}[b]{0.45\linewidth}
      \centering
      \includegraphics[width=0.9\textwidth]{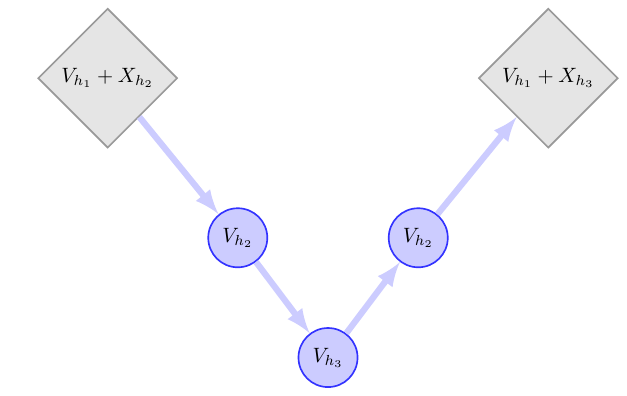}
      \caption{\label{fig:bmg-v}}
\end{subfigure}
\begin{subfigure}[b]{0.8\linewidth}
      \centering
      \includegraphics[width=0.9\textwidth]{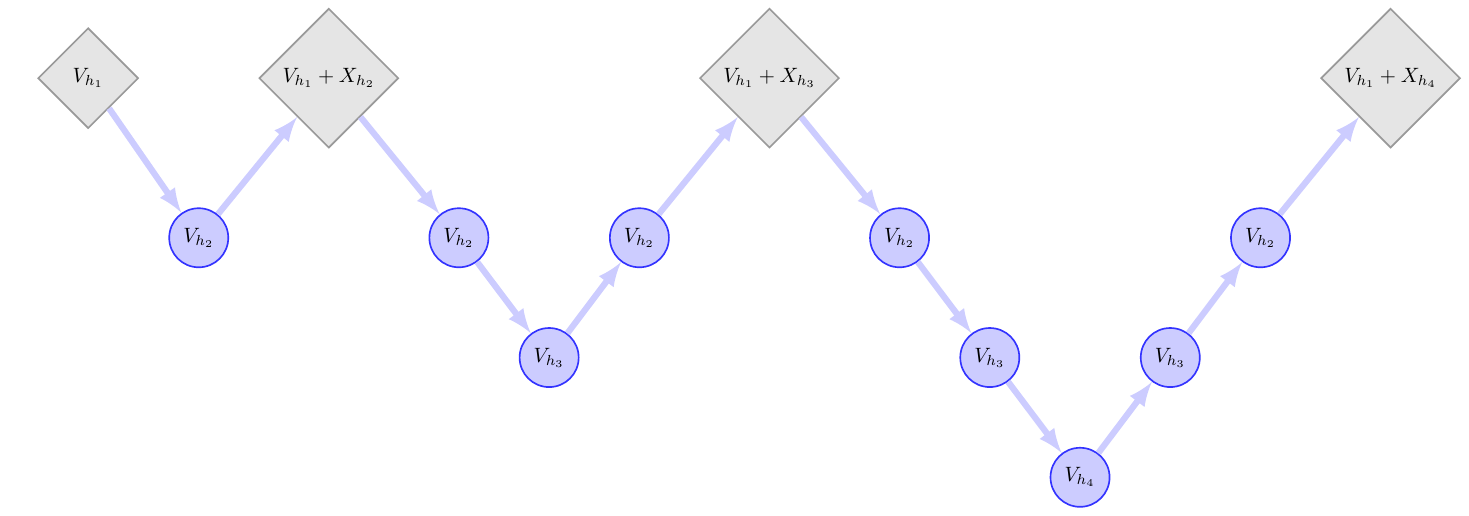}
      \caption{\label{fig:bmg-f}}
\end{subfigure}
\caption{The illustration of Algorithms
\ref{alg:mg-correction}, \ref{alg:mg-vcycle}, 
and \ref{alg:mg-fmg} in (\subref{fig:bmg-t}) , (\subref{fig:bmg-v}) , and (\subref{fig:bmg-f})  respectively. A gray diamond box stands 
for a direct eigensolve on the coarse level (with or without the enrichment), a blue circle stands for a source problem approximation (smoother/solve) on finer levels. The names in the boxes or 
circles stand for the finite element spaces that are used in the various steps
of the algorithm. (\subref{fig:bmg-t}) A two-grid bootstrap algorithm between the coarse and fine levels. (\subref{fig:bmg-v}) A BMG V-cycle iteration between level 1 and level 3. The enrichments space $X_{h_2}$ is updated to $X_{h_3}$ using the approximations in $V_{h_3}$. (\subref{fig:bmg-f}) A BFMG cycle iteration between level 1 and level 4. From coarse to fine, 
Rayleigh quotient iteration \eqref{eq:pb-m-rq} is performed. From fine to coarse, 
smoothing \eqref{eq:pb-m-s} is performed.}
\end{figure}

\begin{algorithm}[ht]
\caption{A BMG V-cycle for approximating eigenpairs between level 1 and 
level $k\geq 2$. \textit{Input:} $(u_{h_{k-1},i},\mu_{h_{k-1}},\lambda_{h_{k-1},i})$, 
\textit{Output:} $(u_{h_k,i},\mu_{h_{k}},\lambda_{h_k,i})$. }
\label{alg:mg-vcycle}
\begin{algorithmic}[1]
\State{\textbf{Coarse grid eigensolve in an enriched space}}
\item[] For a fixed shift $\mu_{h_{k-1}}$, perform a direct 
eigensolve for the coarse grid eigenpairs 
$(u_{h_{k-1},i},\lambda_{h_{k-1},i})\in V_{h_1,h_{k-1}}\times \RR^+$ 
for $i\in \Lambda$, where $V_{h_1,h_{k-1}} = V_{h_1}+X_{h_{k-1}}$, and $\Lambda$ 
is index set of $X_{h_{k-1}}$ on the $(k-1)$-th level, satisfying: for 
any $v\in V_{h_1,h_{k-1}} $,
\begin{equation}
a_{h_1,h_{k-1}} (u_{h_{k-1},i} ,v) 
- \mu_{h_{k-1}} b_{h_1,h_{k-1}}(u_{h_{k-1},i}, v) 
= \lambda_{h_{k-1},i} \, b_{h_1,h_{k-1}} (u_{h_{k-1},i}, v). 
\end{equation}

\State{\textbf{Rayleigh quotient iteration}}
\item[] For $s = 2,\dots, k-1$, approximate for ${u}^{h_s,i}$ in the 
following source problem on level $s$: for any  $v\in V_{h_s}$
\begin{equation}
\label{eq:pb-m-rq}
a_{h_s} ({u}^{h_s,i},v) - \mu_{h_{k-1}}b_{h_s}({u}^{h_s,i}, v) 
=  \lambda_{h_{k-1},i} \,b_{h_s,h_{k-1}} ({u}^{h_{s-1},i}, v),
\end{equation}
where ${u}^{h_{1},i}:= u_{h_{k-1},i} \in V_{h_1}+X_{h_{k-1}}$.

\State{\textbf{Smoothing of fine grid auxiliary source problem}}
\item[] Applying a smoother for ${u}^{h_k,i}\in V_{h_k}$ from level $k$ back to 
level 1 for the following problem defined on the finest grid: for any $v\in V_{h_k}$
\begin{equation}
\label{eq:pb-m-s}
a_{h_k} ({u}^{h_k,i} ,v) - \mu_{h_{k-1}} b_{h_k} ({u}^{h_k,i}, v)= 
\lambda_{h_{k-1},i} b_h (\cP_{h_k} {u}^{h_{k-1},i}, v).
\end{equation}
Then orthogonalized ${u}^{h_{k},i}$'s with respect to the inner product 
$b_{h_k}(\cdot,\cdot)$.

\State{\textbf{Coarse shift and enrichment space updates}} 
\item[]  Update $X_{h_k} = \operatorname{span} \{{u}^{h_k,i}\}_{i\in \Lambda}$. 
Update the coarse grid shift $\mu_{h_{k-1}}$ from the previous cycle to $\mu_{h_k}$ 
based on the Rayleigh quotients of the eigenfunctions in $X_{h_k}$.

\State{\textbf{Coarse grid eigensolve in the updated enriched space and shift}}
\item[] Find $(u_{h_k},\lambda_{h_k})\in V_{h_1,h_k}\times \RR^+$
satisfying, for $i\in \Lambda$, and any $v\in V_{h_1,h_k}$:
\begin{equation}
a_{h_1,h_{k}} (u_{h_{k},i} ,v) 
- \mu_{h_{k}} b_{h_1,h_{k}}(u_{h_{k},i}, v) 
= \lambda_{h_{k},i} \, b_{h_1,h_{k}} (u_{h_{k},i}, v). 
\end{equation}
where $V_{h_1,h_k} = V_{h_1} + X_{h_k}$ is the updated enriched coarse space.
\end{algorithmic}
\end{algorithm}

\subsection{The bootstrap multigrid cycle}
\label{sec:bmg}
The V-cycle variant of Algorithm~\ref{alg:mg-correction} is given by incorporating 
this two-grid method into a multilevel setting, as we outline in 
Algorithm~\ref{alg:mg-vcycle}, and illustrate briefly using 3 levels in 
Figure~\ref{fig:bmg-v}. The algorithm here is a simplified variant of the BAMG approach presented in~\cite{BAMG2010} since
the coarse-level systems and the restriction and interpolation operators
can be defined using the finite element formulation at hand; in BAMG the 
restriction and interpolation operators are defined algebraically and the coarse 
level system is computed using the Galerkin definition.  

For the geometric BMG algorithm in this paper, we construct an 
``almost'' hierarchical sequence of finite element spaces that are geometrically 
projected to the finest 
one, with a total number of $K$ meshes: 
$\wh{V}_{h_1} \subset \wh{V}_{h_2} \subset \dots \subset \wh{V}_{h_{K-1}}\subset 
V_{h_K}$. The meshes are obtained using the surface mesh refinement procedure in Section 
\ref{sec:prelim}, and then finite element spaces on the $1$ through $(K-1)$-th 
levels are projected onto the $K$-th level by recursively using the projection in 
\eqref{eq:proj}. We note that the coarsest space has been enriched by approximations 
from all finer spaces at the end of one BMG V-cycle. 

In contrast 
to the conventional fine-coarse-fine multigrid V-cycle for a source problem, the BMG 
V-cycle (Algorithm \ref{alg:mg-vcycle}) poses itself as an ``inverted'' V-cycle as 
the relaxations process as coarse-fine-coarse. In Figure 
\ref{fig:bmg-v}, the first $V_{h_2}$ node corresponds to the presmoothing stage of 
the conventional multigrid V-cycle. The Rayleigh quotient iteration problem
\eqref{eq:pb-m-rq} is approximated in $V_{h_2}$ using a smoother. Starting from the 
$V_{h_3}$ node, together with the second $V_{h_2}$ node, 1 smoothing for the 
source problem \eqref{eq:pb-m-s} is performed on each level, resembling the 
postsmoothing stage of the conventional multigrid V-cycle. The diamond boxes 
represent the direct eigensolve in the space $V_{h_1} + X_{h_i}$ ($i= 2,3$).

Finally, we present the bootstrap full multigrid (BFMG) method in 
Algorithm \ref{alg:mg-fmg}. The transition from the BMG 
V-cycle (Algorithm \ref{alg:mg-vcycle}) to the BFMG (Algorithm \ref{alg:mg-fmg}) 
resembles the conventional full geometric multigrid cycle that applies V-cycles in 
an incremental fashion, in terms of the levels (meshes) involved. On the coarsest 
mesh, the eigenvalue problem is directly solved
and then the eigenpairs of interest form the right hand sides of the 
source problems that are approximated on the finer meshes. The approximated 
solutions obtained from the source problems are then used to enrich the coarse 
space. Overall, the algorithm is continuously improving the eigenpair approximations 
by working mainly on the coarsest level. The algorithm is again illustrated briefly 
using 4 levels in Figure~\ref{fig:bmg-f}.

For the BFMG algorithm (Algorithm \ref{alg:mg-fmg}), we remark that the 
final output on the $K$-th level, which approximates a certain eigenpair of 
interest, is given 
by the source problem approximation $(u_{h_K},\lambda_{h_K}):= 
({u}^{h_K},{\lambda}^{h_K})$. 
The effect of a single BMG V-cycle for this approximation resembles the two-grid 
method in Algorithm \ref{alg:tg} in that a direct eigensolve is applied in the 
(enriched) coarse space and, then, the source approximation is computed on the finer 
spaces. 

Otherwise, if the aim is to  recover the entire spectrum, then the mesh can be 
continuously refined in which case 
the final output from the BFMG algorithm (Algorithm \ref{alg:mg-fmg}) converges
to the true eigenpairs of problem \eqref{eq:pb-eig-weak}.

\begin{algorithm}[ht]
\caption{BFMG scheme for approximating eigenpairs over 
$K$ levels.}
\label{alg:mg-fmg}
\begin{algorithmic}[1]
\State{\textbf{Coarse eigensolve}}
\item[] Let $(u_{h_1,i},\lambda_{h_1,i})\in V_{h_1}\times \RR^+$ 
($i\in \Lambda \subset \{1,\dots,\dim V_{h_1}\}$)
perform a direct eigensolve for the following problem
\[
a (u_{h_1,i} ,v) - \mu_{h_1} b (u_{h_1,i}, v) = \lambda_{h_1,i} b (u_{h_1,i}, v), 
\quad \text{ for any } v\in V_{h_1}.
\]

\State{\textbf{V-cycle iteration}}
\item[] Set a certain level $k>1$, perform the V-cycle iteration as in Algorithm 
\ref{alg:mg-vcycle} between level 1 and $k$.

\State{\textbf{Rayleigh quotient}} 
\item[] If $k<K$, $k \gets k+1$. If $k=K$, compute the eigenvalue approximation on 
the finest level for all $i\in \Lambda$ using the source approximations.
\end{algorithmic}
\end{algorithm}

\medskip
\section{Numerical Experiments}
\label{sec:n-ex}
In this section, we report various results from the finite element approximation of 
the eigenvalue problem \eqref{eq:pb-eig-weak} on a 2-sphere $\SS^2$.  
The numerical experiments in this section are carried out using the finite 
element toolbox $i$FEM (\cite{Chen.L2008c}). The source codes for these experiments are publicly available at \url{http://github.com/lyc102/ifem/tree/master/research}.
The initial coarse mesh is generated by 
\texttt{Distmesh} (\cite{Persson-distmesh}).

In all tests (of both the two-grid methods (TG) and the 
bootstrap full multigrid method (BFMG) from Algorithm \ref{alg:mg-fmg}), the finer 
meshes are obtained from a uniform refinement of the coarser mesh. The mesh sizes 
satisfy that $h_{k-1}^4 \lesssim h_k$ between the coarser mesh at $(k-1)$-th level 
and the finer mesh at $k$-th level. The newly created vertices are then projected on 
to the continuous surface.

In all the BMG approaches, the dimension of the enrichment space
on the coarsest level $\dim X_h$  
is fixed unless explicitly stated otherwise. 
This dimension is usually set as the 
multiplicity of the largest possible eigenvalue being computed plus some additional 
overlap with its neighboring eigenvalues in the discrete spectra. 
We note that, in the first two subsections, unless 
specifically stated otherwise the source problems on the finer levels are solved 
using a direct method.

The true solutions to the eigenvalue problem \eqref{eq:pb-eig-strong} of the
Laplacian-Beltrami operator on the 2-sphere are known as the real spherical 
harmonics (e.g. see \cite{hobson1955}).
Specifically, the $j$-th eigenvalues, for $l^2\leq j\leq (l+1)^2-1$, counting 
multiplicity, are $\lambda_j = l(l+1)$ for $l\in \ZZ^+$. The dimension of the 
associated eigenspace to the $l$-th distinctive eigenvalue is $2l+1$.

In computation, due to the numbering of the vertices 
in the triangulation, the eigenfunctions obtained approximate a rotated version of 
the spherical harmonics represented using Cartesian coordinates. For this reason, we 
use an a posteriori error estimator to give the error 
estimate of the eigenfunctions under $H^1(\G)$-seminorm. The error 
estimator we use is a combination of the one in \cite{Demlow2007adaptive} for a 
non-eigenvalue problem, and the one from \cite{Duran2003posteriori} for the 
eigenvalue problem on a polygonal domain. The local error estimator for a surface 
triangle $T\in \cT_h$ is defined as follows:
\begin{equation}
\label{eq:est-res}
\begin{aligned}
\eta_T^2 & = h_T^2\lambda_h^2\norm{u_h}_{L^2(T)}^2 
+\frac{1}{2} \sum_{e\subset{\p T}} h_e 
\norm{\jump{\nabla_{\G} u_h \cdot (\vn_T\cross\btau_e)}{e}}_{L^2(e)}^2 
+ \norm{\vB_h \nabla_{\G} u_h}_{L^2(T)}^2.
\end{aligned}
\end{equation}
For notations please refer to \cite{Demlow2007adaptive,Bonito2018priori,BonitoDemlowNochetto2020finite}. The 
global error estimator is obtained by $\eta(u_h) := \left(\sum_{T\in \cT_h} 
\eta_T^2  \right)^{1/2} $.

\subsection{Standard two-grid eigensolver}
\label{sec:n-tg}
To begin we apply the two-grid method from Algorithm~\ref{alg:tg} with shift 
$\mu = 0$ and note that the indefinite source problem on the finer grid is 
solved directly.  Our aim here is 
to illustrate the spectra loss phenomenon that is caused by a coarse space being 
too small. To this end, the coarsest mesh that approximates the sphere is set to 
have only $54$ nodes, i.e., $\#(\texttt{DoF}_H) = 54$. 
The coarsest spectrum computed by \texttt{eigs} can be 
viewed in Figure \ref{fig:cp-eigs}. We can observe that the true $\lambda_j = 30$ 
for $26\leq j \leq 36$, and the dimension of the eigenspace for $\lambda_j = 30$ is 
$11$.   Further, we see that the numerical approximations to the higher end of the 
spectrum have large errors, that is, we see that the approximated $\lambda_{H,j}$ on 
the coarse mesh are closer to the next eigenvalue $42$ in the spectra than the true 
eigenvalue $30$ that they are supposed to approximate. 
This causes part of the eigenpairs from the coarse eigenspace $M(\lambda_{H,j})$ for 
$26\leq j \leq 36$ to converge to eigenpairs of the true $\lambda_{j'} = 42$ 
for $37\leq j' \leq 49$.  As a result, as seen in Figure \ref{fig:cp-eigshtg} 
$\dim(M(\lambda_{h,j})) = 8 < \dim(M(\lambda_{j})) = 11$, where $M(\lambda_{h,j})$ 
is the finite element approximated eigenspace associated with 
$\lambda_{h,j} \approx 30$. 

Note that the spectrum loss becomes even more severe for the approximation 
of the eigenspace for the true eigenvalue $\lambda_j = 42$  ($37\leq j \leq 49$). 
In this case, if we use the same coarse mesh only four eigenfunctions are recovered 
in the fine space (see Figure \ref{fig:cp-eigshtg}).

\begin{figure}[htbp]
  \centering
\begin{subfigure}[b]{0.38\linewidth}
    \centering
    \includegraphics[width=0.9\textwidth]{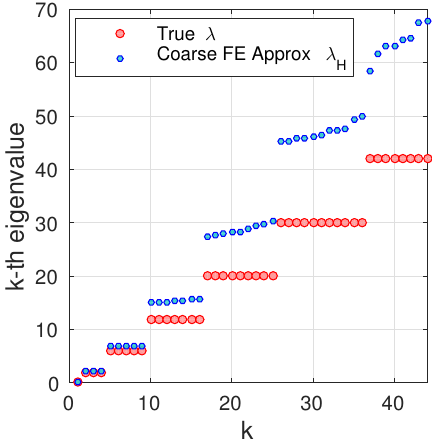}
    \caption{\label{fig:cp-eigs}}
\end{subfigure}%
\hspace{0.3in}
\begin{subfigure}[b]{0.38\linewidth}
      \centering
      \includegraphics[width=0.9\textwidth]{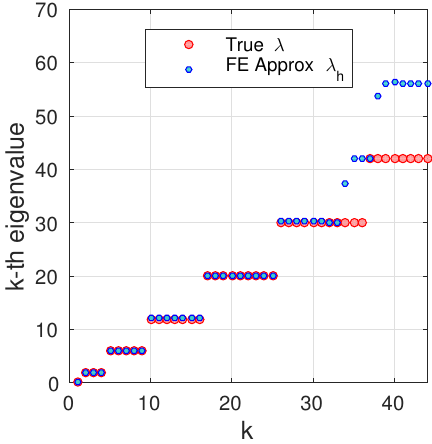}
      \caption{\label{fig:cp-eigshtg}}
\end{subfigure}
\vspace{-0.2in}
\caption{Coarse $\#(\texttt{DoF}_H) = 54$. TG method results in the approximated 
eigenspace dimension being less than intended. Blue dots are the Rayleigh 
quotients computed by the two-grid finite element approximations. Red dots are true 
eigenvalues. (\subref{fig:cp-eigs}) Coarse grid approximation $\lambda_H$'s versus the true $\lambda$. (\subref{fig:cp-eigshtg}) Two-grid approximation $\lambda_h$'s after 5 levels versus the true 
$\lambda$.}
\end{figure}

When a certain eigenpair is ``lost'' in the two-grid approximation scheme, a 
finer coarse mesh can be used in order to recover an improved approximation. This is 
because on a finer coarse mesh the wavelength of these eigenfunctions can be 
resolved. As an example consider the case where the coarse mesh $\#(\texttt{DoF}_H) 
= 54$ and the $34$-th eigenvalue ($\lambda_{34} = 30$) is wrongly approximated by 
the two-grid method (see Figure \ref{fig:cp-eigshtg}). 
If a finer coarse mesh is used instead with $\#(\texttt{DoF}_H) = 96$, then the 
$34$-th eigenvalue is recovered. A comparison of this case is provided in Figure 
\ref{fig:eig-34}.

\begin{figure}[htbp]
  \centering
\begin{subfigure}[b]{0.22\linewidth}
    \centering
    \includegraphics[width=0.9\textwidth]{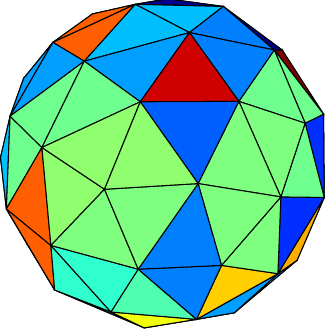}
    \caption{\label{fig:eig-34-1}}
\end{subfigure}%
\hspace{0.1in}
\begin{subfigure}[b]{0.22\linewidth}
      \centering
      \includegraphics[width=0.9\textwidth]{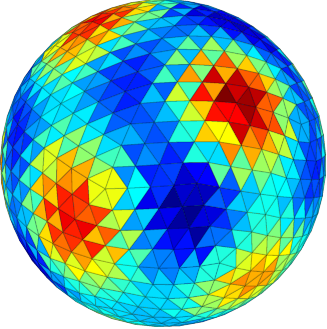}
      \caption{\label{fig:eig-34-2}}
\end{subfigure}
\hspace{0.1in}
\begin{subfigure}[b]{0.22\linewidth}
    \centering
    \includegraphics[width=0.9\textwidth]{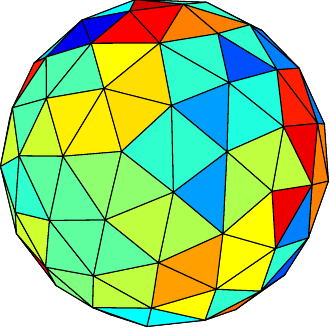}
    \caption{\label{fig:eig-34-3}}
\end{subfigure}%
\hspace{0.1in}
\begin{subfigure}[b]{0.22\linewidth}
    \centering
    \includegraphics[width=0.9\textwidth]{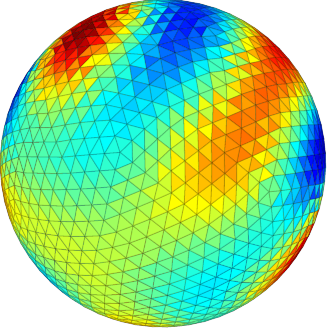}
    \caption{\label{fig:eig-34-4}}
\end{subfigure}
\caption{Two-grid method, comparison of finite element approximations of one 
eigenfunction associated with eigenvalue $\lambda_{34}  = 30$ starting from 
different coarse meshes. (\subref{fig:eig-34-1}) $\#\texttt{DoF}=54$, $u_H$ with $\lambda_{H,34}
\approx 47.5871$. (\subref{fig:eig-34-2}) After 3 levels of refinement, 
$\lambda_{h,34} \approx 397353$, $u_h$ becomes an approximation (both visually and in terms of the Rayleigh quotient) to an eigenfunction associated with $\lambda_{j} = 42$ 
($37\leq j \leq 49$). (\subref{fig:eig-34-3}) $\#\texttt{DoF}=96$, $u_H$ with $\lambda_{H,34} \approx 40.8393$. (\subref{fig:eig-34-4}) After 3 levels of refinement, $\lambda_{h,34} \approx 30.6344$, correct approximation $u_h$ to an eigenfunction associated with 
$\lambda_{j} = 30$ ($26\leq j \leq 36$).}
\label{fig:eig-34}
\end{figure}

\subsection{Bootstrap multigrid enrichment}
\label{sec:n-bm}
An alternative to simply increasing the coarse space by refining the mesh
is given by the bootstrap idea to enrich the coarse space with eigenfunctions
obtained by approximating source problems on finer meshes as in Algorithm 
\ref{alg:mg-fmg}.

If the dimension of the coarsest enriched space is fixed such that
$\dim \mathcal{V}_H + 1$ in all the V-cycles (as in the method from~\cite{Lin2015multi}), then 
approximating larger eigenvalues with multiplicity greater than one still requires a 
finer coarse space. 

To see this we consider the case where the user wants to recover the eigenfunctions
that correspond to $\lambda_j = 42$, where $37\leq j \leq 49$. 
In this example, the coarsest space $\mathcal{V}_H = 
V_{h_1}$ is chosen as in the previous example where $\#(\texttt{DoF}_H) = 
54$. We set the coarse space being enriched as
$V_{h_1,h_2} := V_{h_1}+\{u^{37,h_2} \}$, and solve the coarse  
eigenvalue problem using a direct eigensolve \texttt{eigs} in MATLAB.  The results in 
Figure \ref{fig:cp-eigh1h2} show that the 
new approximation $\lambda_{h_2,37}$ does approximate the $37$-th eigenvalue to some 
extent. However, when we perform this procedure again and then solve the coarse 
eigenvalue problem in the 
updated enriched space $V_{h_1,h_3} := V_{h_1}+\{u^{37,h_3} \}$, we observe
that the approximation to the eigenpair of interest does not improve (see Figure 
\ref{fig:cp-eigh2h3}). 

The reason for this behavior can be 
explained as follows.  First, we see that the approximations to $\lambda_{h_1,j}$ 
($26\leq j\leq 36$) are closer to the true 
eigenvalue $\lambda_j = 42$  ($37\leq j \leq 49$) than the corresponding discrete 
spectra $\lambda_{h_1,j}$ ($37\leq j\leq 49$). As the coarse space is enriched by a 
single function $\{u^{h_2,37}\}$, which is obtained by solving a single source 
problem in $V_{h_k}$ ($k\geq 2$) on the $k$-th level, 
the new approximation $\lambda_{h_k,37}$ becomes closer to the true 
eigenvalue $\lambda_{37} = 42$ that it is supposed to approximate. However, it is 
still not as good an approximation as $\lambda_{h_1,j}$ ($26\leq j\leq 36$) to the 
true eigenvalue $\lambda_{37} = 42$. As such, the algorithm mixes these modes and 
then can not detect the eigenpair of interest.

\begin{figure}[htbp]
  \centering
\begin{subfigure}[b]{0.38\linewidth}
    \centering
    \includegraphics[width=0.9\textwidth]{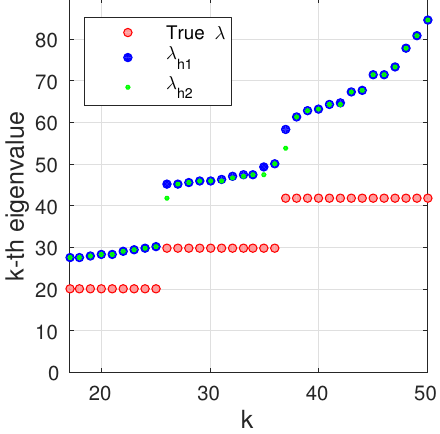}
    \caption{\label{fig:cp-eigh1h2}}
\end{subfigure}%
\hspace{0.3in}
\begin{subfigure}[b]{0.38\linewidth}
      \centering
      \includegraphics[width=0.9\textwidth]{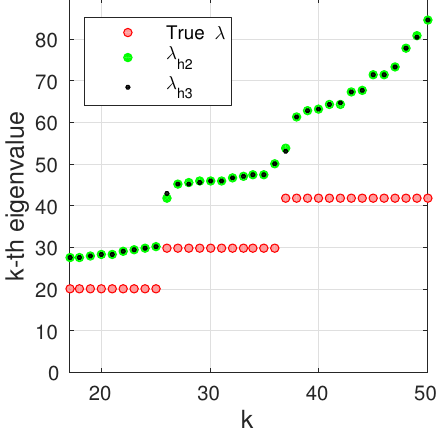}
      \caption{\label{fig:cp-eigh2h3}}
\end{subfigure}
\caption{Coarse $\#(\texttt{DoF}_H) = 54$, single enrichment multigrid method 
results in the eigenpair of interest not being represented in the discrete spectrum. 
Red dots are true eigenvalues. (\subref{fig:cp-eigh1h2}) Coarse eigensolve approximations $\lambda_{h_1}$, and $\lambda_{h_2}$ obtained from solving a coarse eigenvalue problem in an enriched space 
$V_{h_1}+\{\tilde{u}_{h_2} \}$, versus the true $\lambda$. (\subref{fig:cp-eigh2h3}) $\lambda_{h_2}$ obtained from $V_{h_1}+\{\tilde{u}_{h_2}\}$, and 
$\lambda_{h_3}$ obtained from $V_{h_1}+\{\tilde{u}_{h_3}\}$, versus the true 
$\lambda$.}
\end{figure}

Heuristically speaking, for the Laplace-Beltrami eigenvalue problem on $\SS^2$, 
assuming the a priori knowledge of the dimension of the eigenspace, it follows that 
if we seek to approximate the $l$-th distinctive eigenvalue $\lambda_j = l(l+1)$, 
then the coarse space $V_{h_1}$ should at least be enriched by the subspace
$X_{h_k} = \operatorname{span}\{u^{h_k,j'}\}_{j'\in \Lambda }$, where 
$\Lambda = \{ j' \in \ZZ: (l-1)^2+1\leq j' \leq j\}$. 
In this example, for $\lambda_{37} = 42$ where $l=7$, the enrichment space $X_{h_k}$ 
is $\operatorname{span}\{u^{h_k,j'}\}_{26\leq j'\leq 37}$, 
where the $u^{h_k,j'}$'s denote the solutions to the source problems in the 
finer space $V_{h_k}$ on the $k$-th level ($k\geq 2$) for the 26-th to 37-th 
eigenpairs of the discrete spectra. The comparison can be found in Figure 
\ref{fig:cp-eigh1h2bt}, \ref{fig:cp-eigh2h3bt}. Now, because the discrete 
eigenvalues $\lambda_{h_k,j}$ ($26\leq j\leq 36$) are better approximated after the 
coarse space is enriched with multiple eigenfunctions, the algorithm is able to 
better detect the eigenpair of interest $\lambda_{h_k,j}$ ($j=37$, $k=2,3$) after 
the second coarse eigensolve, and overall we see that the approximations improve.

\begin{figure}[htbp]
  \centering
\begin{subfigure}[b]{0.38\linewidth}
    \centering
    \includegraphics[width=0.9\textwidth]{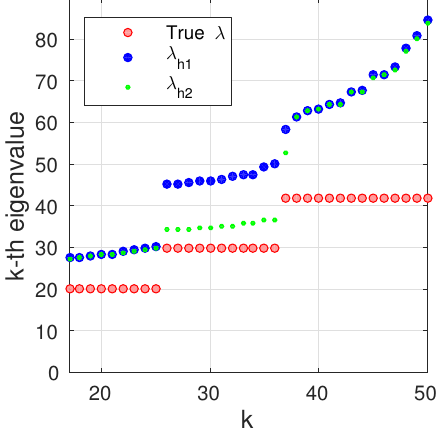}
    \caption{\label{fig:cp-eigh1h2bt}}
\end{subfigure}%
\hspace{0.3in}
\begin{subfigure}[b]{0.38\linewidth}
      \centering
      \includegraphics[width=0.9\textwidth]{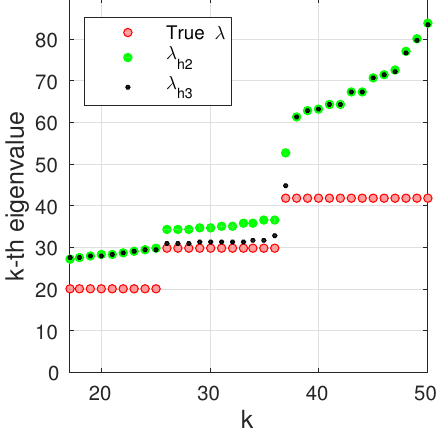}
      \caption{\label{fig:cp-eigh2h3bt}}
\end{subfigure}
\caption{Coarse $\#(\texttt{DoF}_H) = 54$, the enrichment space is chosen according 
to the placement of the eigenpair of interest on the discrete spectra. Red 
dots are true eigenvalues. (\subref{fig:cp-eigh1h2bt}) Coarse grid eigensolve approximations $\lambda_{h_1}$'s, and $\lambda_{h_2}$'s obtained from the eigensolve in the enriched space 
$V_{h_1}+ \operatorname{span}\{u^{h_2,j'}\}_{26\leq j'\leq 37}$, 
versus the true $\lambda$. (\subref{fig:cp-eigh2h3bt}) Enriched coarse eigensolve approximations
$\lambda_{h_2}$'s obtained from 
$V_{h_1}+ \operatorname{span}\{u^{h_2,j'}\}_{26\leq j'\leq 37}$, and 
$\lambda_{h_3}$'s obtained from 
$V_{h_1}+ \operatorname{span}\{u^{h_3,j'}\}_{26\leq j'\leq 37}$, versus the 
true $\lambda$.}
\end{figure}

If one is interested in recovering all the eigenvalues from the smallest one, then  
Algorithm \ref{alg:mg-vcycle} can be applied with an index-fixed $X_{h_k}$, and 
for each distinctive eigenvalue one at a time. Here we choose a fixed dimension 20 
(greater than the biggest numerical multiplicity observed from the coarse 
eigensolve). The 20 enrichment candidate functions are the 
eigenfunctions associated with the 20 eigenvalues nearest to the eigenvalue of 
interest on the discrete spectra, and these eigenfunctions are kept through the BFMG 
cycle in $X_{h_k}$ (level $k=2,3,4$). The discrete spectra recovered using this 
setting of BFMG (Algorithm \ref{alg:mg-fmg}) can be found in Figure 
\ref{fig:cp-eigshmg} and Figure \ref{fig:eigshmg3}.

\begin{figure}[htbp]
  \centering
\begin{subfigure}[b]{0.38\linewidth}
    \centering
    \includegraphics[width=0.9\textwidth]{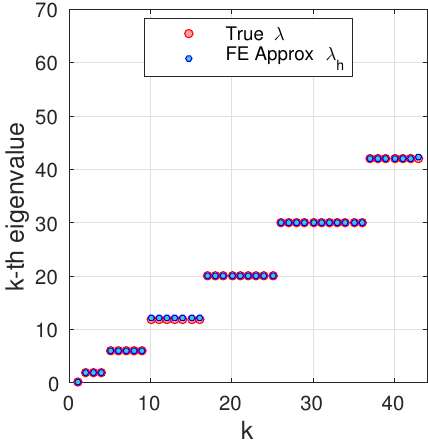}
    \caption{\label{fig:cp-eigshmg}}
\end{subfigure}%
\hspace{0.3in}
\begin{subfigure}[b]{0.38\linewidth}
      \centering
      \includegraphics[width=0.9\textwidth]{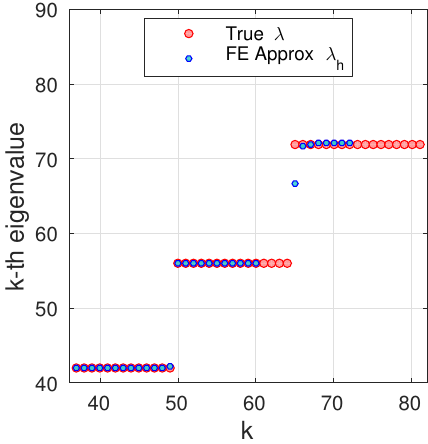}
      \caption{\label{fig:cp-eigshmg17}}
\end{subfigure}
\caption{Coarse $\#(\texttt{DoF}_H) = 54$. BMG with an index-fixed
enrichment space of dimension 20 converges as intended (left). BMG with an 
index-updating enrichment space is able to detect higher eigenvalues that can not be 
resolved in the coarse grid (right). Blue dots are the Rayleigh quotient computed by 
the multigrid finite element approximations. Red dots are true eigenvalues. (\subref{fig:cp-eigshmg})
Multigrid approximations $\lambda_h$'s after 4 levels versus the true 
$\lambda$. (\subref{fig:cp-eigshmg17}) Multigrid approximations $\lambda_h$'s after 6 levels versus the true $\lambda$.}
\end{figure}

\begin{figure}[htbp]
  \centering
\begin{subfigure}[b]{0.38\linewidth}
    \centering
    \includegraphics[width=0.8\textwidth]{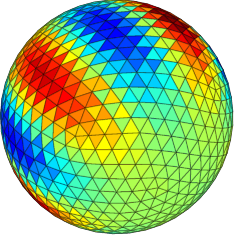}
    \caption{\label{fig:eigshmg-34}}
\end{subfigure}%
\hspace{0.3in}
\begin{subfigure}[b]{0.38\linewidth}
      \centering
      \includegraphics[width=0.8\textwidth]{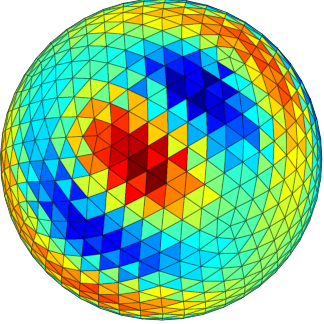}
      \caption{\label{fig:eigshmg-42}}
\end{subfigure}
\caption{The BMG results of the finite element approximations of the 
eigenfunction associated with eigenvalues $\lambda_{34} = 30$ 
and $\lambda_{37} = 42$ after 4 levels of refinement starting with 
$\#\texttt{DoF}=54$ (Figure shows solutions on the 
3rd level). For TG, these approximations can only be recovered correctly with coarse
$\#\texttt{DoF}=96$. (\subref{fig:eigshmg-34})
$\#\texttt{DoF}=3330$, $\lambda_{h_4,34} 
\approx 30.3061$, correct approximation $u_h$ to an eigenfunction associated with 
$\lambda_{j} = 30$ ($26\leq j \leq 36$). (\subref{fig:eigshmg-42}) $\#\texttt{DoF}=3330$, $\lambda_{h_4,42} 
\approx 42.4909$, correct approximation $u_h$ to an eigenfunction associated with 
$\lambda_{j} = 42$ ($37\leq j \leq 49$).}
\label{fig:eigshmg3}
\end{figure}

The other way of improving the eigenvalue approximation for a certain range is to 
set a shift $\mu >0 $ in Algorithm \ref{alg:mg-fmg} to achieve a ``zoom 
in'' effect. In Figure \ref{fig:cp-eigsHs}, a coarse grid shift $\mu_{h_1} = 32$ is 
applied in the coarse solve firstly in Algorithm \ref{alg:mg-correction} involving 
two levels. The 
enrichment candidate space $X_{h_1}$ is chosen to be a fixed dimension of 20, using 
the 20 eigenfunctions associated with the 20 eigenvalues closest to 0. After the 
source solves in $V_{h_2}$, when performing the eigensolve in an enriched coarse 
space $V_{h_1,h_2} = V_{h_1} + X_{h_2}$, the new coarse 
grid shift $\mu_{h_2}$ is 
set to be the previous shift $\mu_{h_1}$ plus the average of the Rayleigh quotients 
of the functions in $X_{h_2}$. $X_{h_2}$ contains all the 
source solutions from problem \eqref{eq:pb-f-aux} using functions from $X_{h_1}$ as 
the sources. The new shift $\mu_{h_2}$ is then used in the  
next BMG V-cycle (Algorithm \ref{alg:mg-vcycle}) involving three levels. This 
averaging for the new shift is 
ad-hoc. The reason for this choice is that we observe in practice that performing 
the BFMG (Algorithm \ref{alg:mg-fmg}), when applied 
continuously among multiple levels, brings the whole discrete spectra closer to the 
true spectra. Consequently, the averaging after each V-cycle 
is to compensate for this change. Figure 
\ref{fig:cp-eigshsmg} contains results obtained with applying this procedure twice, between 
level 3 to level 1 and between level 4 to level 1. The convergence after 5 levels of refinement can be found in 
Figure \ref{fig:conv-s}. We note that if the shift is unchanged after each 
V-cycle, then a similar phenomenon as in Figure \ref{fig:cp-eigh2h3} is observed. 

\begin{figure}[htbp]
  \centering
\begin{subfigure}[b]{0.38\linewidth}
    \centering
    \includegraphics[width=0.9\textwidth]{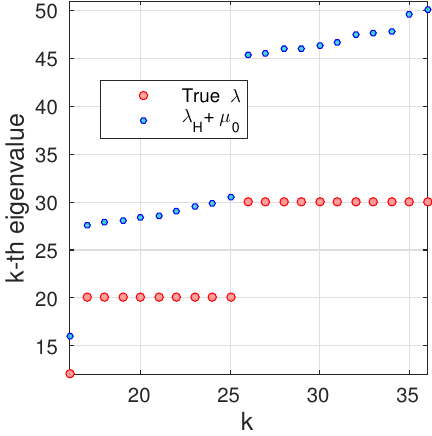}
    \caption{\label{fig:cp-eigsHs}}
\end{subfigure}%
\hspace{0.3in}
\begin{subfigure}[b]{0.38\linewidth}
      \centering
      \includegraphics[width=0.9\textwidth]{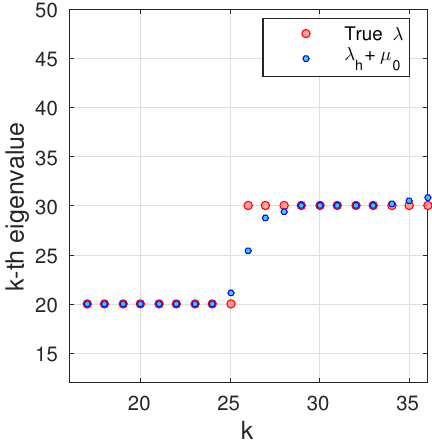}
      \caption{\label{fig:cp-eigshsmg}}
\end{subfigure}
\caption{Coarse $\#(\texttt{DoF}_H) = 54$, both the enrichment space and the shift 
$\mu$ are chosen on each level based on the current approximation to the eigenpair 
of interest on the discrete spectra. Red dots are true eigenvalues.(\subref{fig:cp-eigsHs})
Coarse grid approximation $\lambda_{h_1}$'s, the shift $\mu_{h_1} = 32$. (\subref{fig:cp-eigshsmg}) $\lambda_{h_4}$'s obtained from $V_{h_1}+ X_{h_4}$, the shift 
$\mu_{h_4} = 24.9861$.}
\end{figure}

\begin{figure}[htbp]
  \centering
\begin{subfigure}[b]{0.38\linewidth}
    \centering
    \includegraphics[width=0.9\textwidth]{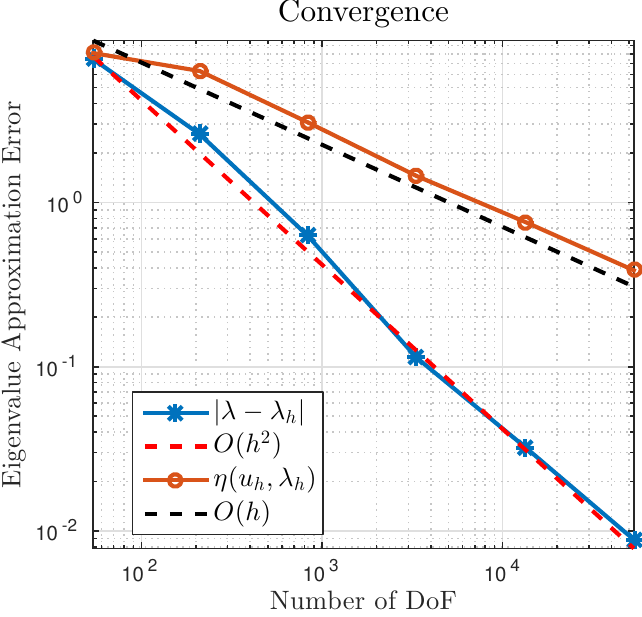}
    \caption{\label{fig:conv-s}}
\end{subfigure}%
\hspace{0.3in}
\begin{subfigure}[b]{0.38\linewidth}
      \centering
      \includegraphics[width=0.9\textwidth]{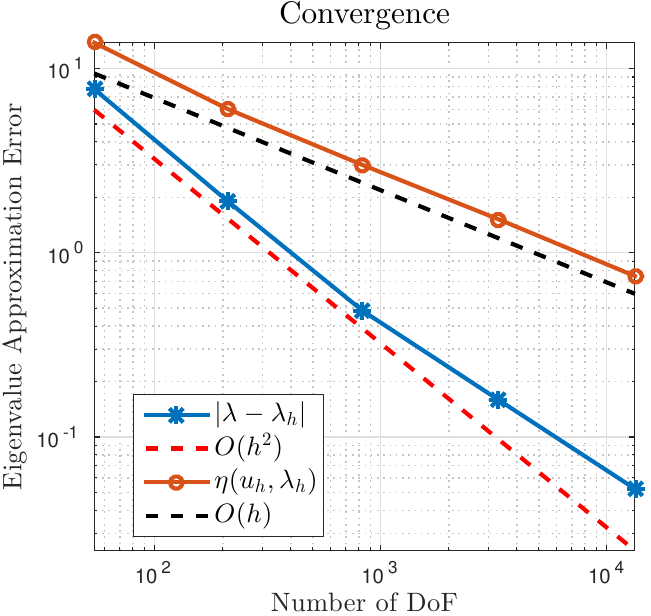}
      \caption{\label{fig:conv-r}}
\end{subfigure}
\caption{Convergence results of BMG using the shifted problem with an updating $\mu$ 
to approximate interior eigenvalue $\lambda= 20$.
The enrichment space's dimension 
is $20$ where the eigenfunctions of the $20$ eigenvalues nearest to 0 are used. The 
coarse eigensolve is performed with a shift updated after each BMG V-cycle 
\ref{alg:mg-vcycle}. The convergence is plotted for the smallest eigenvalue among 
the ones which approximate $\lambda= 20$.(\subref{fig:conv-s}) The source problem on the current level is solved by a direct solver.(\subref{fig:conv-r}) The source problem on the current level is approximated by 5 sweeps of Kaczmarz relaxations.}
\end{figure}

Another interesting observation is that by computing the enrichment space according 
to the eigenpairs that we are attempting to approximate we are able to obtain 
improved approximations to these interior eigenvalues without having to
first compute accurate approximations to the smaller eigenvalues and their
corresponding eigenfunctions (see Figure \ref{fig:cp-eigh2h3bt}).  Thus, the BFMG 
(Algorithm \ref{alg:mg-fmg}) approach can be used to compute interior eigenpairs.  
In contrast, the approach proposed in~\cite{Chan15} only computes multiple 
eigenpairs starting with the smallest.  
In this way, our method for the shifted Laplace-Beltrami eigenvalue problem is more 
general and flexible than previous methods that have been developed 
for the Laplace eigenvalue problem on the plane.
 
\subsection{Solving versus relaxing the source problem}
\label{sec:n-rs}
In previous works on designing multigrid eigensolvers (e.g. 
\cite{Lin2015multi,racheva2002superconvergence}), the auxiliary 
problems in the correction step are solved with a direct method on the fine level. 
In the BMG algorithms (Algorithms \ref{alg:mg-correction}, \ref{alg:mg-vcycle}, 
\ref{alg:mg-fmg}), the exact solve is replaced with an iterative solver (a 
smoother such as symmetric Gauss-Seidel, or Kaczmarz relaxations). We illustrate in 
the numerical experiments that a direct solve of the shifted and indefinite system 
is not necessary, and that relaxation on the shifted source problems
suffices to guarantee the optimal convergence rate for the approximation of 
the eigenpair the algorithm produces. In the first test of this subsection, we 
compare the following 4 
ways to deal with the source problem on the fine level:

\begin{itemize}[topsep=2pt, leftmargin=2.5em]
\item[(1)] TG Algorithm \ref{alg:tg} with shift $\mu=0$ applied in a cascading 
fashion as in Remark \ref{rmk:tg-cascade}. The shift for finer levels comes from 
previous level, direct solve is applied to the shifted (indefinite) problem 
\eqref{eq:pb-f-s} 
on the finer levels when the algorithm is applied between neighbor levels.

\item[(2)] Same setting with (1). For the shifted (indefinite) problem 
\eqref{eq:pb-f-s}, 5 sweeps of Kacmarcz smoother per level is applied using the 
prolongation of the approximation from previous level as initial guess. 

\item[(3)] BFMG Algorithm \ref{alg:mg-fmg} with shift $\mu = 0$. For the unshifted 
positive definite source problems \eqref{eq:pb-f-aux} of BMG on the $k$-th level, 
direct solve is used.

\item[(4)] Same setting with (3). In the BMG V-cycle Algorithm \ref{alg:mg-vcycle}, 
1 sweep of Gauss-Seidel smoother per level is applied to problem \eqref{eq:pb-f-aux} 
using the prolongation of the approximations from previous levels as initial guesses. 
\end{itemize}

Note that the simplification in (2) and (4) reduces the overall computational cost 
of the approach significantly.

To judge whether the source problems need to be 
approximated to the machine precision as in a direct solve in (1) or (3), or some 
sweeps of a smoother suffice in (2) or (4), a key measure is to check the rates 
of convergence $r$, when the meshes are continuously refined. The rate of 
convergence $r$ (notice this is 
different from the rate of convergence for a linear system solver) satisfies the 
following: 
\begin{equation}
|\lambda - \lambda_{h_k}| \sim O(\#(\texttt{DoF}_k)^{-r}) \sim O(h_k^{2r}),
\end{equation}
where $k$ stands for the numbering of the levels, $h_k$ is the mesh size of the 
$k$-th level. $\lambda_{h_k}$ is the approximation of the eigenvalues 
obtained either from a direct eigensolve on level $k$, Algorithm \ref{alg:tg} 
applied cascadingly from level 1 through level $k$, or Algorithm \ref{alg:mg-fmg}. 
Note that, when the mesh is continuously refined such that $H = h_1 > h_2 > \cdots 
>h_k>\cdots$, the optimal rate of convergence of 
$O(h_k^2)$ by a direct solve predicted by the a priori estimate \eqref{eq:est-ev} 
for the eigenvalue approximation is 
achieved with $r\approx 1$. Checking if $r$ for an iterative algorithm is close 
to $1$ in turn implies that the algorithm is convergent in terms 
approximating the true eigenvalues.

In Table \ref{tb:s-r}, we compare the rate of convergence for the approximation of 
the first 3 distinct eigenvalues for the methods (1) through (4) mentioned earlier, 
as the first 3 distinct eigenvalues can be recovered by 
all methods without the loss of spectrum phenomenon in Section \ref{sec:n-tg}.

Firstly, $r$'s for (1) and (2) are compared in the 2nd and 3rd columns in Table 
\ref{tb:s-r}. Then $r$'s for (3) and (4) are compared in the 4th and 5th columns in 
Table \ref{tb:s-r}. 

The experiments correspond to 3 choices of the eigenvalues and the results 
suggest that it is sufficient to solve the unshifted source problem using 
Gauss-Seidel and the BFMG method (Algorithm \ref{alg:mg-fmg}) still yields a nearly 
optimal rate of convergence with 
$r \approx 1$.  In addition, the promising two-grid results obtained using the 
Kaczmarz iteration for the indefinite system together with the multilevel results 
reported in the previous section for the shifted (indefinite) Laplace-Beltrami 
eigenvalue problem suggest that the algorithm will also work well for symmetric 
indefinite problems such as the Helmholtz equation.

\begin{table}[h]
\begin{center}
\begin{tabular}{|c|c|c|c|c|}
\hline $\lambda$ & TG & TG w/ Kaczmarz  & BMG & BMG w/ GS\\ 
\hline   $2$ & $1.0084$ & $0.9656$ & $1.0037$ & $0.9963$\\
\hline   $6$ & $1.0063$ & $0.9575$ & $1.0005$ & $0.9764$\\
\hline   $12$ & $1.0084$ & $0.9731$ & $1.0059$ & $0.9801$ \\
\hline
\end{tabular} 
\vspace{0.1in}
\caption{Comparison of the rates of convergence $r$ as in $O(\#(\texttt{DoF})^{-r})$ 
for the first 3 distinct eigenvalues' approximation of Laplace-Beltrami operator on 
$\SS^2$. }
\label{tb:s-r}
\end{center}
\end{table}

In the second test of this subsection, we compare the following two methods studied 
in the last part of Section \ref{sec:n-bm}:
\begin{itemize}

\item[(5)] BFMG Algorithm \ref{alg:mg-fmg} with an adaptive shift $\mu$ after each 
BMG V-cycle. For the shifted indefinite source problems \eqref{eq:pb-m-s} of on the 
current finest level, direct solve is used.

\item[(6)] Same setting with (5). In the BMG V-cycle Algorithm \ref{alg:mg-vcycle}, 
5 sweeps of Kaczmarz smoother per level is applied to problem 
\eqref{eq:pb-m-s} using the prolongation of the approximation from previous level 
 as initial guess. 
\end{itemize}

The $r$'s for (5) and (6) are compared in the 2nd and 3rd columns in Table 
\ref{tb:s-r-20}. The comparison of the convergence is shown in Figure 
\ref{fig:conv-s} and Figure \ref{fig:conv-r}.

\begin{table}[h]
\begin{center}
\begin{tabular}{|c|c|c|}
\hline $\lambda$ & Shifted BMG & Shifted BMG w/ Kaczmarz \\ 
\hline   $20$ & $0.9861$ & $0.9054$ \\
\hline
\end{tabular} 
\vspace{0.1in}
\caption{Comparison of the rates of convergence $r$ in $O(\#(\texttt{DoF})^{-r})$ 
for $\lambda= 20$. }
\label{tb:s-r-20}
\end{center}
\end{table}

The convergence rates of the BFMG (Algorithm \ref{alg:mg-fmg}) with uniform 
refinement are further verified numerically for the Laplace-Beltrami eigenvalue problem 
on the sphere in Figure~\ref{fig:conv}. These results are on par with the a priori 
TG estimates in \eqref{eq:est-tgs}. Note that for larger eigenvalues,  $e.g., 
\lambda_j=30$ (plot on the right) more degrees of freedom are needed in order to 
obtain the same 
accuracy as for smaller eigenvalues, $e.g., \lambda = 2$ (plot on the left).  Of 
course, the resolution required to achieve high accuracy for large eigenvalues is 
expected to be greater since these modes are generally more oscillatory.  
The key difference in the BMG algorithms (Algorithm 
\ref{alg:mg-correction}, 
\ref{alg:mg-vcycle}, and \ref{alg:mg-fmg}) is that with the enriched coarse space
it is possible to approximate larger eigenvalues without needing to continuously 
increase the size of the coarsest system in order to approximate larger eigenpairs, 
as required by the standard two-grid method (Algorithm \ref{alg:tg}).

\begin{figure}[htbp]
  \centering
\begin{subfigure}[b]{0.38\linewidth}
    \centering
    \includegraphics[width=0.9\textwidth]{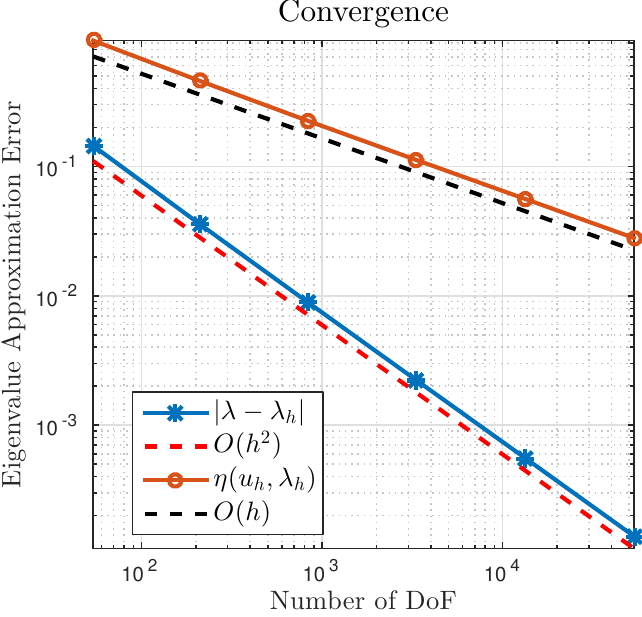}
    \caption{\label{fig:conv-1}}
\end{subfigure}%
\hspace{0.3in}
\begin{subfigure}[b]{0.38\linewidth}
      \centering
      \includegraphics[width=0.9\textwidth]{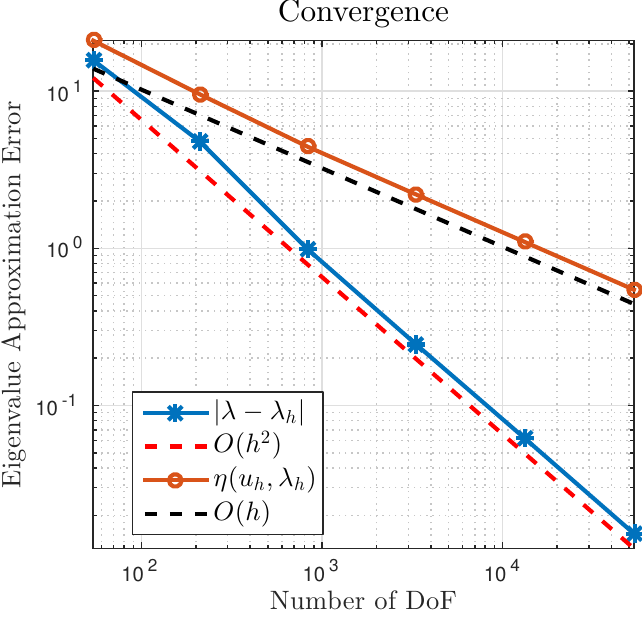}
      \caption{\label{fig:conv-2}}
\end{subfigure}
\caption{Typical convergence results of BMG: using unshifted problem to approximate $\lambda = 2$, 
$\lambda = 30$. (\subref{fig:conv-1}) $\lambda=2$. The enrichment space's dimension 
is fixed to be $6$ using the $6$ eigenfunctions of the lowest 6 eigenvalues from the coarse eigensolve. (\subref{fig:conv-2}) $\lambda=30$. The enrichment space's dimension 
is fixed to be $15$ using the $15$ eigenfunctions of the $15$ eigenvalues starting 
from the $20$-th eigenvalue from 0.}
\label{fig:conv}
\end{figure}

\section{Conclusions}  In this paper, we introduced a bootstrap multigrid method
for solving PDE eigenvalue problems.  As we showed, the BMG algorithm is a geometric version of
the BAMG setup algorithm from~\cite{BAMG2010} and it provides an overall framework
for the design of multigrid eigensolvers.  As an example, we designed and analyzed a finite element 
bootstrap multigrid method for solving the shifted Laplace-Beltrami eigenvalue problem.  
Our analysis and numerical experiments focused on this model problem since it is a challenging 
eigenvalue problem (the eigenvalues have high multiplicity and discretization and refinement of the problem leads 
to a non-nested sequence of meshes in the multilevel hierarchy) and the true solutions are known, allowing us to 
systematically study the performance of the algorithm.  

First, we extended the standard two-grid method from~\cite{HuCheng11ieig} for the Laplace 
eigenvalue problem on the plane to the shifted Laplace-Beltrami eigenvalue problem on a 
surface and we showed that the method gives an optimal $O(h^2)$ convergence, assuming 
that the coarsening is not too aggressive and that the coarse eigenvalue problem is 
solved using a direct method.  We also showed that if the coarse mesh is not 
fine enough then a ``spectra loss'' phenomena occurs and that approximating
larger eigenvalues requires increasingly finer coarse meshes. 

To treat this observed spectral loss in the standard two-grid algorithm, we considered 
the BMG method that enriches the coarse space with a \emph{subspace} 
consisting of increasingly accurate approximations to the eigenpairs of interest.  
Moreover, we showed that these approximations can be computed by applying only a few 
steps of relaxation to related  
symmetric source problems on finer meshes. Moreover, we showed that our approach is able to approximate eigenpairs with large multiplicity 
provided that the enrichment space has a sufficiently large dimension.  
Generally, the dimension of the 
coarsest space needs to depend on the dimension of the eigenspace that the user 
wants to compute. 
In addition, we showed that the bootstrap eigensolver can be used to compute a large portion of 
the eigenspace, starting with the smallest eigenpairs.  We also showed that if 
instead only a few large eigenpairs are sought, then shifted indefinite systems can 
be solved using the BMG algorithm with Kaczmarz smoother in order to compute 
interior eigenpairs directly.  Finally, we proved the convergence of our two-grid BMG 
approach for the shifted Laplace-Beltrami eigenvalue problem in a simplified setting, illustrating the ability of the method to directly compute interior eigenvalues.

\section*{Acknowledgements}
The authors' research was supported in part by the National Science Foundation 
under grants DMS-1320608, DMS-1418934, DMS-1620346, and DMS-2136075. The authors appreciate the referees for their suggested revisions of the paper.

\end{document}